\documentclass[compress, preprint,9pt]{elsarticle}

\usepackage[abs]{overpic}
\usepackage{graphicx}
\usepackage{subfigure}
\usepackage{stmaryrd}
\usepackage{amsfonts, color, amssymb}
\usepackage{amsmath}
\usepackage{amsthm}
\usepackage{epsfig}
\usepackage{psfrag}
\usepackage{bm}
\usepackage{paralist}
\usepackage{algorithm}
\usepackage{algorithmicx}
\usepackage{algpseudocode}
\usepackage{appendix}
\usepackage{caption}
\usepackage{hyperref}
\usepackage{color}

\usepackage[misc]{ifsym}

\usepackage[a4paper, total={6in,8in}]{geometry}

\allowdisplaybreaks[4]
\newtheorem{theorem}{Theorem}[section] %
\newtheorem{definition}{Definition}[section] %

\newtheorem{example}{Example}[section]

\newtheorem{lemma}{Lemma}[section]

\newcommand{\bu}{{\bf u}}
\newcommand{\bw}{{\bf w}}
\newcommand{\be}{{\bf e}}
\newcommand{\bv}{{\mathbf v}}

\def\bbf{{\bf f}}
\def\bg{{\bf g}}

\def\bn{{\bf n}}
\def\3bar{{|\hspace{-.02in}|\hspace{-.02in}|}}

\allowdisplaybreaks[4]

\allowdisplaybreaks 

\numberwithin{equation}{section}

\begin{document}

\begin{frontmatter}

\title{The weak Galerkin finite element method for Stokes interface problems with curved interface}

\author[mymainaddress]{Lin Yang}
\ead{linyang22@mails.jlu.edu.cn}
\author[mymainaddress]{Qilong Zhai\corref{mycorrespondingauthor}}
\cortext[mycorrespondingauthor]{Corresponding author}
\ead{zhaiql@jlu.edu.cn}
\author[mymainaddress]{Ran Zhang}
\ead{zhangran@jlu.edu.cn}

\address[mymainaddress]{School of Mathematics, Jilin University, Changchun 130012, Jilin, P. R. China}

\begin{abstract}
In this paper, we develop a weak Galerkin (WG) finite element scheme for the Stokes interface problems with curved interface. The conventional numerical schemes rely on the use of straight segments to approximate the curved interface and the accuracy is limited by geometric errors. Hence in our method, we directly construct the weak Galerkin finite element space on the curved cells to avoid geometric errors. For the integral calculation on curved cells, we employ non-affine transformations to map curved cells onto the reference element. The optimal error estimates are obtained in both the energy norm and the $L^2$ norm. A series of numerical experiments are provided to validate the efficiency of the proposed WG method.
\end{abstract}


\end{frontmatter}


\section{Introduction}
\label{section:introduction}
In this paper, we focus on the Stokes interface problems. For simplicity, we adopt a specific model to describe the problems. Let $\Omega \subset \mathbb{R}^2$ be an open bounded domain,
which is partitioned into two subdomains, $\Omega_1$ and $\Omega_2$. In each subdomain, the flow is governed by the incompressible Stokes equations, i.e.
\begin{eqnarray}
	-\nabla\cdot(A_i \nabla \bu_i)+\nabla p_i &=&\bbf_i ,\,\, \qquad\text{in}\ \Omega_i,\label{model 1} \\
	\nabla\cdot \bu_i &=&0 , \,\,\,\qquad\text{in}\ \Omega_i,\\
	\bu_i&=&\bg_i , \qquad\text{on}\ \partial \Omega_i \backslash \Gamma ,
\end{eqnarray}
with the viscosity coefficient $A_i > 0$ defined in the $\Omega_i$. 
For simplicity of analysis, let $A_i$ be the piecewise constant matrix in this paper. And  $\Gamma=\partial \Omega_1 \cap \partial\Omega_2$ denotes the interface between two subdomains and belongs to $C^2$ piecewise. The interface conditions on $\Gamma$ are described by the following equations:
\begin{eqnarray}
	\bu_1-\bu_2&=&\bm{\phi}, \qquad\text{on}\ \Gamma, \label{interface 1}\\
	(A_1 \nabla \bu_1-p_1 I)  \mathbf{n}_1 +(A_2 \nabla \bu_2-p_2 I)  \mathbf{n}_2&=&\bm{\psi}, \qquad\text{on}\ \Gamma, \label{interface 2}
\end{eqnarray}
where $\mathbf{n}_1$ and $\mathbf{n}_2$ are the unit outward normal vectors on $\Gamma$. $\bn_1$ points from $\Omega_{1}$ into $\Omega_{2}$ and $\bn_2=-\bn_1$ (see Figure \ref{figure1}) and $I$ represents the identity matrix. We shall drop the subscript $i$ when velocity function $\bu$ and pressure function $p$ are defined in the whole domain $\Omega$. The Stokes interface problems are classical fluid mechanics problems that describe the flow of viscous fluid through interfaces. The interface problems have been widely applied in different fields such as groundwater resource management, petroleum engineering, biomedicine, and more \cite{Stokesinterfacebackground4,Stokesinterfacebackground1,Stokesinterfacebackground6,Stokesinterfacebackground2,Stokesinterfacebackground5,Stokesinterfacebackground3}.

In practical problems, the interface is usually a complex curved surface. In the two-dimensional problems, when dealing with such curved interface, a common approach involves approximating these curves with straight line segments. However, when employing high-degree polynomials to approximate the exact solution, the presence of geometric errors leads to a reduction in the orders of convergence \cite{reduceerror3,reduceerror1,reduceerror2}. Therefore, in domain with curved edges, various numerical methods have been proposed to solve the problems. For example, finite element method \cite{FEMcurved1,pFEM,isogeometricmethods,NEFEM,reduceerror3}, discontinuous Galerkin finite element method\cite{DGcurved2,DGcurved3,DGcurved1,DGcurved4,DGcurved5}, virtual element method \cite{VEMcurved2,VEMcurved1}, weak Galerkin finite element method \cite{guan2019weak,Wang_Curved_edges,mu_weak_2021}, etc.

In this work, we use the weak Galerkin (WG) finite element method to solve the Stokes interface problems with curved edges on fitted meshes. The WG method was first proposed in \cite{wang2013weak} for solving second order elliptic equations. Compared with other methods, the method uses separate polynomial functions on each cell and adds stabilizers to the cell boundaries to ensure weak continuity of the approximative functions. Since polynomial spaces are easy to construct, the WG method can be applied to polygonal or polyhedral meshes. At the same time, this method uses weakly defined differential operators to replace the classical differential operators. The WG method is used to solve various problems, such as Stokes equations \cite{WGStokes1,WGStokes2}, Brinkman equations \cite{WGBrinkman1,WGBrinkman2}, linear elasticity equations \cite{elasticityequation}, parabolic equations\cite{parabolic}, elliptic interface problems \cite{ellipticinterface}, Stokes-Darcy problems \cite{Wang_Stokes_Darcy,stokesDarcy}, etc.

In this paper, we mainly treat curved cells. For the curved cells, instead of replacing curved edges with straight segments, we directly construct the weak Galerkin space on the curved cells. For the integral calculation on the curved cells, we use non-affine transformation\cite{nonaffineisoparametric1,nonaffineisoparametric2} to convert curved cell to the reference element. This treatment avoids geometric errors and makes the scheme easier to implement.

The outline of this paper is as follows: In Section 2, some notations used in this paper are presented. We also give the definitions of the weak Galerkin finite element space and weak differential operators. Based on these definitions, the weak Galerkin finite element scheme is established. Moreover, the proof of the existence and uniqueness is given. Section 3 is devoted to the proof of the stability of the solution of the WG scheme and presents some important inequalities. In Section 4 and 5, the error estimates in the energy norm and the $L^2$ norm are proved separately. Finally, we give some numerical examples to verify our proposed theories in Section 6.

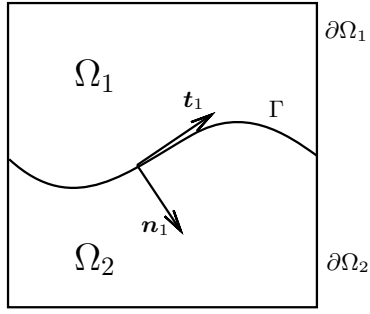
\begin{figure}
	\centering
	\setlength{\unitlength}{1bp}%
	\begin{picture}(134.40, 125.67)(0,0)
		\put(0,0){\includegraphics{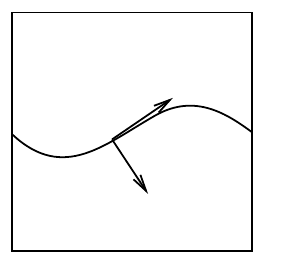}}
		\put(30.28,88.91){\fontsize{14.23}{17.07}\selectfont $\Omega_1$}
		\put(70.96,81.04){\fontsize{8.54}{10.24}\selectfont $\bm{t}_1$}
		\put(55.30,33.85){\fontsize{8.54}{10.24}\selectfont $\bm{n}_1$}
		\put(30.28,19.83){\fontsize{14.23}{17.07}\selectfont $\Omega_2$}
		\put(102.93,77.28){\fontsize{8.54}{10.24}\selectfont $\Gamma$}
		\put(123.79,106.86){\fontsize{8.54}{10.24}\selectfont $\partial \Omega_{1}$}
		\put(124.48,19.83){\fontsize{8.54}{10.24}\selectfont $\partial \Omega_{2}$}
	\end{picture}
	\caption{The model of Stokes-Stokes coupling flow}
	\label{figure1}
\end{figure}
\section{The Weak Galerkin Numerical Scheme}
\label{sec2}
\label{sec:1}
In this section, we first give the introduction of notations used in the paper. Then, we define the weak Galerkin finite element space and the corresponding weak differential operators. Next, the weak Galerkin finite element scheme is established. Based on the scheme, the proof of existence and uniqueness are given.
\subsection{The notations}
\label{sec:2}
Assume that $\mathcal{T}_h^1$ and $\mathcal{T}_h^2$ are the polygon partitions of the domain $\Omega_1$ and $\Omega_2$ containing both straight and curved cells. Set $\mathcal{T}_h = \mathcal{T}_h^1 \cup \mathcal{T}_h^2$. For $T \in \mathcal{T}_h$, when all edges of $T$ are straight, it is called a straight cell. The set of such straight cells is denoted by $\mathcal{T}_h^S$. When $T$ has a curved edge which intersects with the interface, it is called a curved cell. The set of the curved cells is denoted by $\mathcal{T}_h^I$. For $T \in \mathcal{T}_h$, $|T|$ and $h_T$ are the area and the diameter of the cell $T$, respectively. Denote that $h=\max_{T \in \mathcal{T}_h} h_T$ is the mesh size. Let $\mathcal{E}_h$ be the set of all edges on the partition $\mathcal{T}_h$ and $\mathcal{E}_h^0$ be the set of all interior edges including interface edges. Denote by $\mathcal{E}_h^S$ and $\mathcal{E}_h^I$ the set of all straight edges and the set of all curved edges, respectively. The cell $T \in \mathcal{T}_h$ should satisfy regularity conditions in \cite{WGStokes1}. In addition to these conditions, the following conditions also need to be satisfied.

(A1) There exists a positive number $C_1$ satisfying 
\begin{eqnarray}
	|e| \leqslant C_1 h_T,
\end{eqnarray}
for all edges $e \in \partial T$ and  $T \in \mathcal{T}_h$.

(A2) For every cell $T$, there exists a ball in the interior of $T$.
\subsection{Weak Galerkin finite element space}
For $T \in \mathcal{T}_h$, we define the weak function $\bv= \{\bv_0, \bv_b \}$ on the cell $T$, where $\bv_0$ represents the interior function in $T$ and $\bv_b$ represents the boundary function on $\partial T$. Note that $\bv_b$ has one unique value on the edge $e \in \mathcal{E}_h^S$ and there are two values $\bv_{1b}$ and $\bv_{2b}$ on the edge $e \in \mathcal{E}_h^I$. And there is no relationship between $\bv_0$ and $\bv_b$. Next, we give definitions of some projection operators. For a given integer $k \geqslant 1$, let $Q_0$ be the $L^2$ projection operator from $[L^2(T)]^2$ onto $[P_k(T)]^2$, $T \in \mathcal{T}_h$. $Q_b$ is defined as the $L^2$ projection operator from $[L^2(e)]^2$ onto $[P_{k-1}(e)]^2$ when $e \in \mathcal{E}_h^S$, and from $[L^2(e)]^{2}$ onto $[P_{k}(e)]^{2}$ when $e \in \mathcal{E}_h^I$. Finally, we define the weak Galerkin finite element space with respect to the velocity function $\bu$ and the pressure function $p$.
\begin{align*}
	V_h=&\{\bv=\{ \bv_0, \bv_b\}: \bv_0|_{T} \in [P_k(T)]^2, \,T \in \mathcal{T}_h,\\
	&\bv_b|_e \in [P_{k-1}(e)]^2, e \in \mathcal{E}_h^S,
	\bv_b|_e \in [P_k(e)]^2, e \in \mathcal{E}_h^I
	\},\\
	V_h^0=&\{\bv=\{ \bv_0, \bv_b\} \in V_h, \bv_b=\mathbf{0} ~on ~ \partial \Omega, \bv_{1b} - \bv_{2b}= \mathbf{0}, e \in \mathcal{E}_h^I\},\\
	W_h=&\{q: q \in L_0^2(\Omega), q|_T \in P_{k-1}(T),\, T \in \mathcal{T}_h\}.
\end{align*}

Firstly, we introduce some weak differential operators used in this paper.
\begin{definition}\cite{WGStokes1}
	For $\bv \in V_h$, its discrete weak gradient $\nabla_w \bv \in $
	$ [P_{k-1}(T)]^{2\times 2}$ satisfies  
	\begin{equation}
		(\nabla_w \bv, \tau)_T=-(\bv_0,\nabla \cdot \tau)_T + \langle \bv_b,\tau \cdot \bn \rangle_{\partial T},~\forall~\tau \in [P_{k-1}(T)]^{2\times 2},\label{weak gradient}
	\end{equation}
	where $\bn$ is the unit outward normal vector on $\partial T$.
\end{definition}

\begin{definition}\cite{WGStokes1}
	For every $\bv \in V_h$, its discrete weak divergence $\nabla_w \cdot \bv \in P_{k-1}(T)$ satisfies 
	\begin{equation}
		(\nabla_w \cdot \bv,\varphi)_T=-(\bv_0, \nabla \varphi)_T +  \langle \bv_b, \varphi \bn \rangle_{\partial T},~ \forall~\varphi \in P_{k-1}(T), \label{weak divergence}
	\end{equation}
	where $\bn$ is the unit outward normal vector on $\partial T$.	
\end{definition}

\subsection{The numerical scheme}
We define bilinear forms as follows: 
\begin{eqnarray*}
	\begin{split}
		s(\bv,\bw)=&\sum_{T \in \mathcal{T}_{h}} A h_T^{-1}\langle Q_b \bv_0 -\bv_b,Q_b \bw_0 -\bw_b \rangle_{\partial T \setminus (\partial T \cap \Gamma)}\\
		&+\sum_{T \in \mathcal{T}_{h}} A h_T^{-1}  \langle\bv_0 - \bv_b, \bw_0-\bw_b \rangle_{\partial T \cap \Gamma},\\
		a(\bv,\bw)=&\sum_{T \in \mathcal{T}_h}(A\nabla_w \bv,\nabla_w \bw)_T,\\
		a_s(\bv,\bw)=& a(\bv,\bw)+s(\bv,\bw),\\
		b(\bv,p)=&-\sum_{T \in \mathcal{T}_h}(\nabla_w \cdot \bv,p)_T,
	\end{split}
\end{eqnarray*}
for all $\bv$, $\bw \in V_h$ and $p \in W_h$.
\begin{algorithm}
	\caption{Weak Galerkin Scheme}
	For the Stokes interface problems (\ref{model 1})-(\ref{interface 2}), the WG scheme is seeking $\mathbf{u}_h \in V_h$ and $p_h \in W_h$ to satisfy $\bu_b=Q_b \bg ~\text{on} \,  \partial \Omega$, $\bv_{1b} - \bv_{2b}= Q_b \bm{\phi} ~\text{on}~ e \in \mathcal{E}_h^I$, and
	\begin{eqnarray}
		a_s(\bu_h,\bv_h)+b(\bv_h,p_h)&=&(\bbf,\bv_0)+\langle \bm{\psi} ,\bv_b \rangle_{\Gamma}, ~\forall~\bv_h \in V_h^0,\label{algorithm 1}\\
		b(\bu_h,q_h)&=&0,~~~~~~~~~~~~~~~~~~~~~~\forall ~q_h \in W_h. \label{algorithm 2}
	\end{eqnarray}
\end{algorithm}

We now define a semi-norm in the space $V_h$ as follows:
\begin{eqnarray}
	\3bar\bv\3bar^2=a_s(\bv,\bv).
\end{eqnarray}
Then we have the following properties.

\begin{lemma}
	$\3bar \cdot \3bar$ provides a norm in $V_h^0$.
\end{lemma}
\begin{proof}
	Assume that $\3bar \bv \3bar =0$ with $\bv \in V_h^0$, then we obtain
	\begin{eqnarray*}
		\begin{split}
			0=a_s(\bv,\bv)=&\sum_{T \in \mathcal{T}_h}(A\nabla_w \bv,\nabla_w \bv)_T+\sum_{T \in \mathcal{T}_{h}}A h_T^{-1}\langle \bv_0 - \bv_b, \bv_0-\bv_b \rangle_{\partial T \cap \Gamma}\\
			&+\sum_{T \in \mathcal{T}_{h}}A h_T^{-1}  \langle Q_b \bv_0 -\bv_b,Q_b \bv_0 -\bv_b \rangle_{\partial T \backslash (\partial T \cap \Gamma)},		
		\end{split}
	\end{eqnarray*}
	which leads to 
	\begin{eqnarray*}
		\nabla_w \bv = \mathbf{0} ~  in ~ T \in \mathcal{T}_h ,\quad
		Q_b \bv_0=\bv_b ~ on ~  e \in \mathcal{E}_h^S,\quad
		\bv_0 =\bv_b ~  on ~  e \in \mathcal{E}_h^I.
	\end{eqnarray*}
	Therefore, for any $\tau \in [P_{k-1}(T)]^{2 \times 2}$, we have the following results.
	
	For $ T \in \mathcal{T}_{h}^S$, according to the definition (\ref{weak gradient}), integration by parts and the property of the $L^2$ projection operator, we obtain 
	\begin{align*}
		0=&(\nabla_w \bv ,\tau)_T\\
		=&-(\bv_0,\nabla \cdot \tau)_T+\langle \bv_b ,\tau \bn \rangle_{\partial T}\\
		=&(\nabla \bv_0, \tau)_T-\langle \bv_0 ,\tau \bn \rangle_{\partial T}+\langle \bv_b ,\tau \bn \rangle_{\partial T}\\
		=&(\nabla \bv_0, \tau)_T-\langle Q_b \bv_0 -\bv_b ,\tau \bn \rangle_{\partial T}\\
		=&(\nabla \bv_0, \tau)_T.
	\end{align*}
	
	Similarly, for $T \in \mathcal{T}_{h}^I$, we get 
	\begin{align*}
		0=&(\nabla_w \bv ,\tau)_T\\
		=&-(\bv_0,\nabla \cdot \tau)_T+\langle \bv_b ,\tau \bn \rangle_{\partial T}\\
		=&(\nabla \bv_0,\tau)_T-\langle \bv_0 ,\tau \bn \rangle_{\partial T}+\langle \bv_b ,\tau \bn \rangle_{\partial T}\\
		=&(\nabla \bv_0,\tau)_T-\langle Q_b\bv_0 -\bv_b ,\tau \bn \rangle_{\partial T \backslash (\partial T \cap \Gamma)} -\langle \bv_0 -\bv_b ,\tau \bn \rangle_{\partial T \cap \Gamma} \\
		=&(\nabla \bv_0,\tau)_T.
	\end{align*}
	Choosing $\tau=\nabla \bv_0$ in the above equations derives $\nabla \bv_0=0$ in $T \in \mathcal{T}_{h}$. It follows that $\bv_0$ is constant in $T \in \mathcal{T}_{h} $. Due to $\bv_0=Q_b \bv_0 =\bv_b $ on $e \in \mathcal{E}_h$, we have $\bv_0 = \bv_b =C$. According to the fact that $\bv_b=0 $ on $\partial \Omega$, we obtain that $\bv_0=0$ and $\bv_b=0$.
\end{proof}

\begin{lemma}\label{Abounded} \cite{WGStokes1}
	For any $ \bv,\bw \in V_h$, there holds
	\begin{eqnarray*}
		|a_s(\bv,\bw)|&\leqslant& \3bar \bv \3bar \cdot \3bar \bw \3bar.
	\end{eqnarray*}
\end{lemma}

\begin{lemma}
	The WG scheme (\ref{algorithm 1})-(\ref{algorithm 2}) has a unique solution.
\end{lemma}
\begin{proof}
	Since (\ref{algorithm 1})-(\ref{algorithm 2}) are finite-dimensional square linear equations, existence and uniqueness are equivalent. Let $\{ \bu_h,p_h\}$ and $\{ \bu_h',p_h'\}$ be the solutions of the WG scheme (\ref{algorithm 1})-(\ref{algorithm 2}), respectively. For $v_h \in V_h^0$ and $q_h \in W_h$, we have
	\begin{eqnarray}
		a_s(\bu_h , \bv_h)+b(\bv_h , p_h)&=&(\bbf, \bv_0)+\langle \bm{\psi} , \bv_b \rangle_{\Gamma},\label{uniqueness1}\\
		b(\bu_h ,q_h)&=&0,\label{uniqueness2}\\
		a_s(\bu_h' , \bv_h)+b(\bv_h , p_h')&=&(\bbf, \bv_0)+\langle \bm{\psi}, \bv_b \rangle_{\Gamma},\label{uniqueness3}\\
		b(\bu_h' ,q_h)&=&0.\label{uniqueness4}		
	\end{eqnarray}
	Subtracting Eq.(\ref{uniqueness3}) from Eq.(\ref{uniqueness1}) and subtracting Eq.(\ref{uniqueness4}) from Eq.(\ref{uniqueness2}) derives that
	\begin{eqnarray}
		a_s(\bu_h-\bu_h',\bv_h)+b(\bv_h,p_h-p_h') &=& 0, ~ v_h \in V_h^0 \label{uniqueness5}\\
		b(\bu_h-\bu_h',q_h)&=&0,~ q_h \in W_h.\label{uniqueness6}
	\end{eqnarray}
	Taking $\bv_h =\bu_h - \bu_h' \in V_h^0$ in Eq. (\ref{uniqueness5}) and $q_h = p_h - p_h' \in W_h$ in Eq.(\ref{uniqueness6}) leads to
	\begin{eqnarray}
		a_s(\bu_h-\bu_h',\bu_h-\bu_h')+b(\bu_h-\bu_h',p_h-p_h')=0,\label{uniqueness7}\\
		b(\bu_h-\bu_h',p_h-p_h')=0.\label{uniqueness8}
	\end{eqnarray}
	Since $\3bar \cdot \3bar$ is a norm in $V_h^0$, we obtain $\bu_h =\bu_h'$ and $b(\bv_h,p_h - p_h')=0$. Then taking $\bv_h=\{\nabla (p_h - p_h'),\mathbf{0} \}$ gives 	
	$$0 = \sum_{T \in \mathcal{T}_h} \| \nabla (p_h -p_h')\|_T^2,$$
	which leads to
	$p_h-p_h'$ is constant in $ T \in \mathcal{T}_h$.
	
	Next, setting $\bv_0 = \mathbf{0}$ in $T \in \mathcal{T}_h$, $\bv_b = [ p_h - p_h'] \bn_e$ on $e \in \mathcal{E}_h^S$ and $\bv_b =\mathbf{0}$ on $e \in \mathcal{E}_h^I$ in Eq.(\ref{uniqueness5}) leads to 
	$$\sum_{e \in \mathcal{E}_h^S} \|  [ p_h - p_h'] \bn_e \|_e^2 =0.$$ 
	The jump $[v]$ is defined as
	$$
	[v]=v ~ \text{if} ~ e \subset \partial \Omega, \quad [v]=v|_{T_1} -v|_{T_2}, ~ \text{if} ~e \in \mathcal{E}_h^0,
	$$
	where $e$ is the common edge of $T_1$ and $T_2$.
	
	Moreover, when $e \in \mathcal{E}_h^I$, define the other two edges of the same cell as $e_1$ and $e_2$ and their unit outward normal vectors as $\bn_{e_1}$ and $\bn_{e_2}$. Choosing $\bv_0 = {\bf{0}}$ on $T \in \mathcal{T}_h$, $\bv_b ={\bf{0}}$ on $e \in \mathcal{E}_h^S$ and $\bv_b = [p_h -p_h'] (\bn_{e_1}+\bn_{e_2})$ on $e \in \mathcal{E}_h^I$ to derive that
	$$0= [p_h -p_h']_e^2 \langle \bn_{e_1}+\bn_{e_2},\bn_e\rangle_e.$$
	Since $\langle (\bn_{e_1}+\bn_{e_2}),\bn_e\rangle_e < 0$, we get $[ p_h -p_h']_e=0$. Noting that $p_h \in L_0^2(\Omega)$, we have $p_h - p_h' = 0$ on $\Omega$. The proof of the lemma is complete.  
\end{proof}
\section{Stability}
For $T \in \mathcal{T}_h$, let $Q_h=\{Q_0,~Q_b\}$ be the $L^2$ projection operator onto $V_h$. $\mathcal{Q}_h$ and $\mathbb{Q}_h$ are defined as the $L^2$ projection operators onto $P_{k-1}(T)$ and $[P_{k-1}(T)]^{2 \times 2}$, respectively.
\begin{lemma}\label{weak gradient exchange}
	For $\bu \in [ H^1(\Omega)]^2$ and $ \tau \in [P_{k-1}(T)]^{2 \times 2}$, on every cell $T \in \mathcal{T}_h$, we have the following properties for the discrete weak gradient operator.
	
	(1)~For $T \in \mathcal{T}_h^S $ , we have 
	\begin{eqnarray}\label{weak gradient exchange 1}
		(\nabla_w(Q_h \bu),\tau)_T=(\mathbb{Q}_h(\nabla \bu),\tau)_T.
	\end{eqnarray}
	
	(2)~For $T \in \mathcal{T}_h^I$, we have 
	\begin{eqnarray}\label{weak gradient exchange 2}
		(\nabla_w(Q_h \bu),\tau)_T=(\mathbb{Q}_h(\nabla \bu),\tau)_T+\langle  Q_b \bu -\bu , \tau \bn\rangle_{\partial T \cap \Gamma}.
	\end{eqnarray}
\end{lemma}

\begin{proof}For $T$ $\in \mathcal{T}_h^S$ and $\tau \in [P_{k-1}(T)]^{2 \times 2}$, according to Eq.(\ref{weak gradient}), definitions of $Q_0$, $Q_b$ and $\mathbb{Q}_h$ and integration by parts, we have
	\begin{align*}
		(\nabla_w(Q_h \bu), \tau)_T=&-(Q_0 \bu, \nabla \cdot \tau)_T+\langle Q_b \bu, \tau \cdot \bn \rangle_{\partial T}\\
		=&-(\bu,\nabla \cdot \tau)_T+\langle \bu,\tau \cdot \bn \rangle_{\partial T}\\
		=&(\nabla \bu,\tau)_T\\
		=&(\mathbb{Q}_h(\nabla \bu),\tau)_T.
	\end{align*}
	Similarly, for $T$ $\in \mathcal{T}_h^I$ and $\tau \in [P_{k-1}(T)]^{2 \times 2}$, we have
	\begin{align*}
		(\nabla_w(Q_h \bu), \tau)_T=&-(Q_0 \bu, \nabla \cdot \tau)_T+\langle Q_b \bu, \tau \cdot \bn \rangle_{\partial T}\\
		=&-(\bu,\nabla \cdot \tau)_T+\langle \bu ,\tau \cdot \bn \rangle_{\partial T}+\langle Q_b \bu - \bu ,\tau \cdot \bn \rangle_{\partial T \cap \Gamma}\\
		=&(\nabla \bu,\tau)_T+\langle Q_b \bu - \bu ,\tau \cdot \bn \rangle_{\partial T \cap \Gamma}\\
		=&(\mathbb{Q}_h(\nabla \bu),\tau)_T+\langle Q_b \bu - \bu ,\tau \cdot \bn \rangle_{\partial T \cap \Gamma}.
	\end{align*}
	The proof of Lemma \ref{weak gradient exchange} is complete.
\end{proof}

\begin{lemma}\label{weak divergence exchange}
	For $\bu \in [ H^1(\Omega)]^2$ and $\varphi \in P_{k-1}(T)$, on every cell $T \in \mathcal{T}_h$,  we have the following properties for the discrete weak divergence operator.
	
	(1)~For $T \in \mathcal{T}_h^S $, we obtain
	\begin{eqnarray}\label{weak divergence exchange 1}
		(\nabla_w \cdot (Q_h \bu),\varphi)_T=(\mathcal{Q}_h(\nabla \cdot \bu),\varphi)_T.
	\end{eqnarray}
	
	(2)~For $T \in \mathcal{T}_h^I $, we obtain
	\begin{eqnarray}\label{weak divergence exchange 2}
		(\nabla_w \cdot (Q_h \bu),\varphi)_T=(\mathcal{Q}_h(\nabla \cdot \bu),\varphi)_T+\langle  Q_b \bu -\bu , \varphi \bn\rangle_{\partial T \cap \Gamma}.
	\end{eqnarray}
\end{lemma}

\begin{proof}For $T \in \mathcal{T}_h^S$ and $\varphi \in P_{k-1}(T)$, according to Eq.(\ref{weak divergence}), definitions of $Q_0$, $Q_b$ and $\mathcal{Q}_h$ and integration by parts , we get 
	\begin{align*}
		(\nabla_w \cdot (Q_h \bu),\varphi)_T=&-(Q_0 \bu, \nabla \varphi)_T+\langle Q_b \bu ,\varphi \bn \rangle_{\partial T}\\
		=&-(\bu,\nabla \varphi)_T+\langle \bu ,\varphi \bn \rangle_{\partial T}\\
		=&(\nabla \cdot \bu,\varphi)_T\\
		=&(\mathcal{Q}_h(\nabla \cdot \bu),\varphi)_T.
	\end{align*}
	Similarly, for $T \in \mathcal{T}_{h}^I$ and $\varphi \in P_{k-1}(T)$, using Eq.(\ref{weak divergence}) to derive
	\begin{align*}
		(\nabla_w \cdot (Q_h \bu),\varphi)_T=&-(Q_0 \bu,\nabla \varphi)_T+\langle Q_b \bu, \varphi \bn
		\rangle_{\partial T}\\
		=&-(\bu,\nabla \varphi)_T+\langle \bu,\varphi \bn \rangle_{\partial T}+\langle Q_b \bu -\bu ,\varphi \bn \rangle_{\partial T \cap \Gamma}\\
		=&(\nabla \cdot \bu,\varphi)_T+\langle Q_b \bu -\bu ,\varphi \bn \rangle_{\partial T \cap \Gamma}\\
		=&(\mathcal{Q}_h(\nabla \cdot \bu),\varphi )_T+\langle Q_b \bu -\bu ,\varphi \bn \rangle_{\partial T \cap \Gamma}.
	\end{align*}
	The proof of the above lemma is complete.
\end{proof}
Now, we give some important inequalities for the proof. The trace inequality and the inverse inequality are essential technique tools for the analysis. For the straight triangular cells, these inequalities have been proved in \cite{elliptic_mix}. Next we extend two inequalities to the curved cells.
\begin{lemma}(Inverse Inequality)
	For all $T \in \mathcal{T}_h^I$, $\varphi$ is the piecewise polynomial on $T$, then we have
	\begin{eqnarray}
		\|\nabla \varphi \|_T \leqslant C h_T^{-1} \|\varphi\|_T.
	\end{eqnarray}
\end{lemma}
\begin{proof}
	For any $T \in \mathcal{T}_h^I$, assume that $S(T)$ is the circumscribed simplex satisfying the shape regularity conditions. According to the standard inverse inequality \cite{elliptic_mix} on $S(T)$, we have
	\begin{eqnarray}
		\|\nabla \varphi \|_T \leqslant \| \nabla \varphi \|_{S(T)} \leqslant Ch_{S(T)}^{-1} \|\varphi \|_{S(T)}. \label{ProofI1}
	\end{eqnarray}
	Next, consider a ball, denoted as $S$, which is contained entirely within $T$ and has a diameter that is proportional to $h_T$. Then by the domain inverse inequality \cite{elliptic_mix}, we get
	\begin{eqnarray}\label{Inverse1}
		\|\varphi\|_{S(T)} \leqslant \|\varphi\|_{S} \leqslant \|\varphi\|_T.
	\end{eqnarray}
	Substituting inequality (\ref{Inverse1}) into the inequality (\ref{ProofI1}) leads to
	\begin{eqnarray*}
		\|\nabla \varphi\|_{T} \leqslant Ch_{T}^{-1} \|\varphi\|_T.
	\end{eqnarray*}
	The proof of the inverse inequality on the curved cells is complete.
\end{proof}
\begin{lemma}(Trace Inequality)
	For any $T \in \mathcal{T}_h^I$ and $\varphi \in H^1(T)$, we have
	\begin{eqnarray}\label{trace inequality}
		\|\varphi\|_e^2 \leqslant C\left( h_T^{-1} \|\varphi\|_T^2 + h_T \|\nabla \varphi\|_T^2 \right),~~e \subset \partial T.
	\end{eqnarray}
\end{lemma}
\begin{proof}
	For simplify, we consider curved triangular cells. Similarly, the trace inequality holds true for polygons. For $T \in \mathcal{T}_h^I$, assume that the non-affine transformation between $T$ and the reference element $\hat{T}$ is as follows
	\begin{eqnarray*}
		\left \{\begin{array}{rl}
			x=\tilde{x}(\xi, \eta),& (\xi, \eta) \in \hat{T}\\
			y=\tilde{y}(\xi, \eta),& (\xi, \eta) \in \hat{T}
		\end{array}\right.,
	\end{eqnarray*}
	and denote that $\tilde{\varphi}(\xi, \eta)=\varphi(\tilde{x}(\xi, \eta),\tilde{y}(\xi, \eta))$. $\tilde{J}$ is the corresponding Jacobian of the above mapping. Suppose that the edge $e$ possesses parametric representations given by $x = x(t)$ and $y = y(t)$, $0 \leqslant t \leqslant 1$. According to the regularity conditions and Theorem 1 in \cite{nonaffineisoparametric1}, every term of Eq.(\ref{trace inequality}) has the following estimate.  
	\begin{eqnarray*}
		\begin{split}
			\|\varphi\|_e^2 =&\int_e \varphi^2(x, y)ds\\
			=&\int_0^1 \varphi^2(x(t), y(t)) \vert e \vert dt \\
			=&\int_0^1 \tilde{\varphi}^2(\xi(t),\eta(t))\vert e \vert dt\\
			\leqslant & Ch_T\|\tilde{\varphi}\|_{\hat{e}}^2,
		\end{split}
	\end{eqnarray*}
	and
	\begin{eqnarray*}
		\begin{split}
			\|\varphi\|_T^2=&\int_T \varphi(x,y)^2 dT\\
			=&\int_{\hat{T}} \tilde{\varphi}^2 (\xi, \eta) \vert \tilde{J} \vert d\hat{T}\\
			\geqslant &C h_T^2 \int_{\hat{T}} \tilde{\varphi}^2 (\xi, \eta) d\hat{T}\\
			\geqslant &C h_T^2  {h^{-1}_{\hat{T}}}  \|\tilde{\varphi}\|_{\hat{T}}^2,
		\end{split}
	\end{eqnarray*}
	and
	\begin{eqnarray*}
		\begin{split}
			\|\nabla \varphi \|_T^2=&\int_T \left( \frac{\partial \varphi}{\partial x}(x,y)\right)^2 + \left( \frac{\partial \varphi}{\partial y}(x,y)\right)^2 dT\\
			=&\int_{\hat{T}} 
			\left( \frac{\partial \varphi}{\partial x}(\tilde{x}(\xi, \eta),\tilde{y}(\xi, \eta))\right)^2+\left(\frac{\partial \varphi}{\partial y}(\tilde{x}(\xi, \eta),\tilde{y}(\xi, \eta))\right)^2 |\tilde{J}| d \hat{T}\\
			\geqslant&\int_{\hat{T}} \left(\frac{\partial \varphi}{\partial x}(\tilde{x}(\xi, \eta),\tilde{y}(\xi, \eta))\right)^2+\left(\frac{\partial \varphi}{\partial y}(\tilde{x}(\xi, \eta),\tilde{y}(\xi, \eta))\right)^2 h_T^2 d \hat{T}\\
			\geqslant& C \int_{\hat{T}} \left( \frac{\partial \varphi}{\partial x} \frac{\partial \tilde{x}}{\partial \xi} +\frac{\partial \varphi}{\partial y} \frac{\partial \tilde{y}}{\partial \xi}\right)^2 +
			\left( \frac{\partial \varphi}{\partial x} \frac{\partial \tilde{x}}{\partial \eta} +\frac{\partial \varphi}{\partial y} \frac{\partial \tilde{y}}{\partial \eta}\right)^2 d \hat{T}\\
			=&\int_{\hat{T}}  \left(\frac{\partial \tilde{\varphi}}{\partial \xi}\right)^2+
			\left(\frac{\partial \tilde{\varphi}}{\partial \eta}\right)^2 d \hat{T}\\
			\geqslant & C h_{\hat{T}} \|\nabla \tilde{\varphi} \|_{\hat{T}}^2.
		\end{split}
	\end{eqnarray*}
	
	Therefore, according to the trace inequality \cite{elliptic_mix} on the reference element, we have 
	\begin{eqnarray*}
		\begin{split}
			\|\varphi\|_e^2 \leqslant & C h_T\|\tilde{\varphi}\|_{\hat{e}}^2\\ \leqslant& C h_T h_{\hat{T}}^{-1} \|\tilde{\varphi}\|_{\hat{T}}^2+C h_T h_{\hat{T}} \|\nabla \tilde{\varphi}\|_{\hat{T}}^2\\
			\leqslant& C h_T^{-1} \|\varphi\|_{T}^2 + Ch_T  \|\nabla \varphi\|_{T}^2 .
		\end{split}
	\end{eqnarray*}
	The proof of the trace inequality is complete.
\end{proof}
\begin{lemma}
	For $\bw \in [H^{l+1}(\Omega)]^2$, $\rho \in H^l(\Omega)$ with 
	$1 \leqslant l \leqslant k$ and $0 \leqslant s \leqslant 1$, we have the following estimates
	\begin{eqnarray}
		\sum_{T \in \mathcal{T}_h} h_T^{2s} \| \bw- Q_0 \bw\|_{T,s}^2 &\leqslant& Ch^{2(l+1)}\|\bw\|_{l+1}^2,\label{projectorestimate1}\\
		\sum_{T \in \mathcal{T}_h} h_T^{2s} \|\nabla \bw- \mathbb{Q}_h(\nabla \bw) \|_{T,s}^2 &\leqslant& C h^{2l}\|\bw\|_{l+1}^2,\label{projectorestimate2}\\
		\sum_{T \in \mathcal{T}_h} h_T^{2s} \|\rho- \mathcal{Q}_h \rho \|_{T,s}^2 &\leqslant& Ch^{2l}\|\rho\|_{l}^2 \label{projectorestimate3}.
	\end{eqnarray}
	Here $C$ represents a positive constant that remains independent of both the mesh size $h$ and the functions involved in the above estimates.
\end{lemma}
\begin{proof}
	For $T \in \mathcal{T}_h^S$, the inequalities (\ref{projectorestimate1})-(\ref{projectorestimate3}) are same as the inequalities in \cite{WGStokes1}. Therefore we only consider the curved cell $T \in \mathcal{T}_h^I$. Assume that $S(T)$ is the circumscribed simplex satisfying the shape regularity conditions. And let $u$ smoothly extend onto $S(T)$. Denote by $\tilde{Q}_0$ the $L^2$ projection operator onto $[P_k(S(T))]^2$. Then we have
	\begin{eqnarray*}
		\begin{split}
			\sum_{T \in \mathcal{T}_h}\|u-Q_0 u\|_T^2 
			\leqslant & C \sum_{T \in \mathcal{T}_h} \|u-\tilde{Q}_0 u\|_{T}^2 \\
			\leqslant& C \sum_{T \in \mathcal{T}_h} \|u-\tilde{Q}_0 u\|_{S(T)}^2\\
			\leqslant& C h_{S(T)}^{2(r+1)} \|u\|_{r+1,S(T)}^2 \\
			\leqslant &C h_{T}^{2(r+1)} \|u\|_{r+1,S(T)}^2. 
		\end{split}
	\end{eqnarray*}
	From the regularity conditions, the number of overlaps of circumscribed simplex sets is fixed. Then we derive
	\begin{eqnarray*}
		\sum_{T \in \mathcal{T}_h} \|u-Q_0 u\|_T^2 \leqslant Ch^{2(r+1)} \|u\|_{r+1}^2.
	\end{eqnarray*}
	Similarly, we get
	\begin{eqnarray*}
		\sum_{T \in \mathcal{T}_h} h_T^{2s}\|\nabla (u-Q_0 u)\|_{T,s}^2 \leqslant Ch^{2(r+1)} \|u\|_{r+1}^2.
	\end{eqnarray*}
	Therefore, the proof of Eq.(\ref{projectorestimate1}) is complete. The proof of Eqs.(\ref{projectorestimate2})-(\ref{projectorestimate3}) is quite similar to the Eq.(\ref{projectorestimate1}) and so is omitted.
\end{proof}

Next we give the stability analysis of the Stokes interface problems.
\begin{lemma}(Inf-Sup Condition)\label{InfSupCondition}
	There are two positive constants  $C_1$ and $C_2$ to satisfy the following inequality
	\begin{eqnarray}\label{InfSup}
		\sup_{\bv \in V_h^0}\frac{b(\bv,\rho)}{\3bar \bv \3bar} \geqslant C_1 \| \rho\|- C_2 h|\rho|_1,
	\end{eqnarray}
	for all $\rho \in W_h $, where $C_1$ and $C_2$ are independent of the mesh size $h$.
\end{lemma}
\begin{proof}
	For $\rho \in W_h$, according to \cite{bookfiniteelementmethod1,bookfiniteelementmethod2,bookfiniteelementmethod3,bookfiniteelementmethod5,bookfiniteelementmethod4,WGStokes1}, we can find a function $\bv \in [H_0^1(\Omega)]^2$ to satisfy
	\begin{align*}
		\nabla \cdot \bv =-\rho,
	\end{align*}
	and $\| \bv \|_{1} \leq C \| \rho \|$. Here $C > 0$ is a constant depending on the domain $\Omega$. Let $\tilde{\bv}= Q_h \bv \in V_h^0$, then we need to verify the following inequality holds true:
	\begin{eqnarray}\label{InfSup1}
		\3bar \tilde{\bv} \3bar \leqslant C_0 \| \bv \|_1.
	\end{eqnarray}
	First, for $T \in \mathcal{T}_h^S$, it follows from the definition of $\3bar \cdot \3bar$ and Eq.(\ref{weak gradient exchange 1}) that 
	\begin{eqnarray}\label{InfSup1.1}
		A\| \nabla_w \tilde{\bv} \|_T^2 =A\| \nabla_w (Q_h \bv) \|_T^2 = A\|\mathbb{Q}_h(\nabla \bv) \|_T^2 \leqslant A\| \nabla \bv \|_T^2 \leqslant C
		\| \bv \|_{1,T}^2.
	\end{eqnarray}
	Next, using the trace inequality, the definition of $Q_h$ and the estimate (\ref{projectorestimate1}), we derive
	\begin{eqnarray}\label{SvvS}
		\begin{split}
			s(\tilde{\bv},\tilde{\bv})|_T
			=&Ah_T^{-1}\|Q_b (Q_0 \bv) -Q_b \bv \|_{\partial T}^{2}\\
			\leqslant &C h_T^{-1}\|Q_0 \bv-\bv \|_{\partial T}^{2}\\
			\leqslant &C\left( h_T^{-2}\|Q_0 \bv-\bv \|_T^2+\|\nabla(Q_0 \bv-\bv) \|_T^2 \right)\\
			\leqslant &C\| \bv \|_{1,T}^2.
		\end{split}
	\end{eqnarray}
	For $T \in \mathcal{T}_{h}^I$, by Eq.(\ref{weak gradient exchange 2}), the Cauchy-Schwarz inequality and the trace inequality, we have
	\begin{eqnarray}
		\begin{split}
			\| \nabla_w \tilde{\bv} \|_T^2 
			=&(\mathbb{Q}_h(\nabla \bv), \nabla_w \tilde{\bv})_T + \langle Q_b \bv - \bv, (\nabla_w \tilde{\bv})\bn\rangle_{\partial T \cap \Gamma}\\
			\leqslant &
			\| \mathbb{Q}_h (\nabla \bv) \|_T \| \nabla_w \tilde{\bv} \|_T+\|Q_b \bv -\bv \|_{\partial T \cap \Gamma} 
			\| \nabla_w \tilde{\bv} \|_{\partial T \cap \Gamma}\\
			\leqslant & \| \nabla \bv \|_T \| \nabla_w \tilde{\bv} \|_T+ C \|Q_0 \bv -\bv \|_{\partial T \cap \Gamma}\| \nabla_w \tilde{\bv} \|_{\partial T \cap \Gamma}\\
			\leqslant & C \| \bv\|_{1,T} \| \nabla_w \tilde{\bv}\|_T,
		\end{split}
	\end{eqnarray}
	i.e.
	\begin{eqnarray}
		\| \nabla_w \tilde{\bv} \|_T \leqslant C \| \bv \|_{1,T}.
	\end{eqnarray}
	For stabilization on $T \in \mathcal{T}_h^I$, according to the property of the $L^2$ projection operator, the trace inequality and the estimate (\ref{projectorestimate1}), we get
	\begin{eqnarray}\label{InfSup1.2}
		\begin{split}
			s(\tilde{\bv},\tilde{\bv})|_T=&A h_T^{-1}\|Q_b(Q_0 \bv)-Q_b \bv \|_{\partial T \backslash (\partial T \cap \Gamma)}^2+A h_T^{-1}\|Q_0\bv - Q_b\bv \|_{\partial T \cap \Gamma}^2\\
			\leqslant&C h_T^{-1} \|Q_0 \bv-\bv\|^2_{\partial T \backslash (\partial T \cap \Gamma)}+C h_T^{-1}\| Q_0\bv - \bv \|^2_{\partial T \cap \Gamma}+C h_T^{-1} \|\bv -Q_b \bv \|^2_{\partial T \cap \Gamma}\\
			\leqslant & C\| \bv \|_{1,T}^2.
		\end{split}
	\end{eqnarray}
	Consequently, combining inequalities (\ref{InfSup1.1})-(\ref{InfSup1.2}) yields inequality (\ref{InfSup1}). 
	
	The next step in the proof is to estimate $b(\bv,\rho)$. For all $\rho \in W_h$, based on Lemma \ref{weak divergence exchange} and the definition of $\mathcal{Q}_h$, it's easy to see that 
	\begin{eqnarray}\label{InfSup2}
		\begin{split}
			&-\sum_{T \in \mathcal{T}_h} (\nabla_w \cdot (Q_h \bv),\rho)_T\\
			=& - \sum_{T \in \mathcal{T}_h} (\mathcal{Q}_h(\nabla \cdot \bv),\rho)_T - \sum_{e \in \Gamma} \langle Q_b \bv -\bv, \rho \bn \rangle_e\\
			=&-\sum_{T \in \mathcal{T}_h} (\nabla \cdot \bv,\rho)_T -\sum_{e \in \Gamma} \langle Q_b \bv -\bv, \rho \bn \rangle_e\\
			=&\|\rho\|^2-\sum_{e \in \Gamma} \langle Q_b\bv-\bv,
			\rho \bn-Q_b(\rho \bn) \rangle_e\\
			\geqslant& \| \rho \|^2 - Ch \| \bv \|_1 |\rho|_1,
		\end{split}
	\end{eqnarray}
	where we have used the fact that $\sum_{e \in \Gamma} \langle Q_b \bv -\bv, Q_b(\rho \bn) \rangle_e=0$. Combining inequality (\ref{InfSup1}) with inequality (\ref{InfSup2}), we have
	$$
	\frac{b(\bv,p)}{\3bar \bv \3bar} \geqslant \frac{\| \rho \|^2 - Ch \| \bv \|_1 |\rho|_1}{C_0  \|\bv \|_1}
	\geqslant C_1 \Vert \rho \Vert - C_2 h\vert \rho \vert_1.
	$$
	This indicates that inequality (\ref{InfSup}) holds true.
\end{proof}
\section{Error Equations}In this section, we present the error equations for $\bu$ and $p$. We use $(\bu_h , p_h)$ to represent the numerical solutions obtained from the WG scheme. At the same time, denote by $(\bu , p)$ the exact solutions of Eqs.(\ref{model 1})-(\ref{interface 2}). The errors associated with $\bu$ and $p$ are defined as follows:
\begin{eqnarray}
	e_h=Q_h \bu - \bu_h, \quad \varepsilon_h=\mathcal{Q}_h p-p_h.
\end{eqnarray}

\begin{lemma}\label{EE}
	For $(\bu_i, p_i)\in [H^1(\Omega_i)]^2 \times L^2(\Omega_i)$ with $i=1,2$ and $p \in L_0^2(\Omega)$ satisfying Eqs. (\ref{model 1})-(\ref{interface 2}), we derive the following equation:
	\begin{eqnarray}\label{EEE}
		\begin{split}
			a(Q_h \bu, \bv)+b(\bv, \mathcal{Q}_h p)
			=&(\bbf,\bv_0)+\sum_{e \in \Gamma} \langle \bm{\psi}, \bv_b  \rangle_e+\ell_1(\bu,\bv)
			-\ell_2(p,\bv)+\ell_3(\bu,\bv),
		\end{split}
	\end{eqnarray}
	where
	\begin{eqnarray}
		\ell_1(\bu,\bv)&=&\sum_{T \in \mathcal{T}_h} \langle \bv_0 -\bv_b, A \nabla \bu \cdot \bn-A \mathbb{Q}_h(\nabla \bu) \cdot \bn \rangle_{\partial T},\\
		\ell_2(p,\bv)&=&\sum_{T \in \mathcal{T}_h} \langle \bv_0 -\bv_b, p\bn-(\mathcal{Q}_h p) \bn \rangle_e,\\
		\ell_3(\bu, \bv)&=&\sum_{i=1}^2 \sum_{e \in \Gamma} \langle Q_b \bu_i -\bu_i,A_i \nabla_w \bv_i \cdot \bn_i  \rangle_e.
	\end{eqnarray}
\end{lemma}

\begin{proof}
	It follows from Lemma \ref{weak gradient exchange}, the definition of the weak gradient operator and integration by parts that
	\begin{eqnarray}\label{EEP1}
		\begin{split}
			&(\nabla_w (Q_h \bu),A \nabla_w \bv)\\
			=&\sum_{i=1}^2 \sum_{T \in \mathcal{T}_h^i} (\nabla_w (Q_h \bu_i),A_i \nabla_w \bv_i)_T\\
			=&\sum_{i=1}^2 \sum_{T \in \mathcal{T}_h^i}(\mathbb{Q}_h(\nabla \bu_i),A_i \nabla_w \bv_i)_T
			+\sum_{i=1}^2\sum_{T \in \mathcal{T}_h^i} \langle Q_b \bu_i -\bu_i,A_i \nabla_w \bv_i \cdot \bn_i \rangle_{\partial T \cap \Gamma}\\
			=&\sum_{i=1}^2 \sum_{T \in \mathcal{T}_h^i} \left(-(\bv_{i,0},\nabla \cdot (A_i \mathbb{Q}_h (\nabla \bu_i)))_T+\langle \bv_{i,b},
			A_i \mathbb{Q}_h (\nabla \bu_i) \cdot \bn_i \rangle_{\partial T} \right)\\
			&+\sum_{i=1}^2 \sum_{T \in \mathcal{T}_h} \langle Q_b \bu_i -\bu_i,A_i \nabla_w \bv_i \cdot \bn_i \rangle_{\partial T \cap \Gamma}\\
			=&\sum_{i=1}^2 \sum_{T \in \mathcal{T}_h^i} \left( (\nabla \bv_{i,0} ,A_i \nabla \bu_i )_T - \langle \bv_{i,0} -\bv_{i,b},
			A_i \mathbb{Q}_h (\nabla \bu_i) \cdot \bn \rangle_{\partial T} \right)\\
			&+\sum_{i=1}^2 \sum_{T \in \mathcal{T}_h} \langle Q_b \bu_i -\bu_i, A_i \nabla_w \bv_i \cdot \bn_i \rangle_{\partial T \cap \Gamma}.\\
		\end{split}
	\end{eqnarray}
	Similarly, using the definition of the weak divergence operator and integration by parts, we have
	\begin{eqnarray}\label{EEP2}
		\begin{split}
			&-(\nabla_w \cdot \bv, \mathcal{Q}_h p )\\
			=&-\sum_{i=1}^2\sum_{T \in \mathcal{T}_h^i} (\nabla_w \cdot \bv_i, \mathcal{Q}_h p_i )_T\\
			=&\sum_{i=1}^2\sum_{T \in \mathcal{T}_h^i} (\bv_{i,0},\nabla(\mathcal{Q}_h p_i ))_T-\langle \bv_{i,b},(\mathcal{Q}_h p_i) \bn_i \rangle_{\partial T}\\
			=&\sum_{i=1}^2\sum_{T \in \mathcal{T}_h^i} -(\nabla \cdot \bv_{i,0},\mathcal{Q}_h p_i)_T+\langle \bv_{i,0} -\bv_{i,b}, (\mathcal{Q}_h p_i)\bn_i \rangle_{\partial T}\\
			=&\sum_{i=1}^2\sum_{T \in \mathcal{T}_h^i} -(\nabla \cdot \bv_{i,0},p_i)_T+\langle \bv_{i,0} -\bv_{i,b}, (\mathcal{Q}_h p_i)\bn_i \rangle_{\partial T}.
		\end{split}
	\end{eqnarray}
	Then integrate with $\bv_0$ in $\bv=\{ \bv_0, \bv_b\} \in V_h^0$ on two sides of Eq.(\ref{model 1}), we obtain
	\begin{eqnarray}\label{EEP3}
		-(\nabla \cdot (A \nabla \bu),\bv_0)+(\nabla p, \bv_0)=(\bbf, \bv_0).
	\end{eqnarray}
	Using integration by parts, we get
	\begin{eqnarray}\label{EEP4}
		\begin{split}
			&-(\nabla \cdot (A \nabla \bu),\bv_0)\\
			=&\sum_{i=1}^2 \sum_{T \in \mathcal{T}_h^i} (-\nabla \cdot (A_i \nabla \bu_i),\bv_{i,0})_T\\
			=&\sum_{i=1}^2 \sum_{T \in \mathcal{T}_h^i}(A_i\nabla \bu_i, \nabla \bv_{i,0})-\sum_{T \in \mathcal{T}_h} \langle A_i\nabla \bu_i \cdot \bn_i ,\bv_{i,0} \rangle_{\partial T}\\
			=&\sum_{i=1}^2 \sum_{T \in \mathcal{T}_h^i} \left((A_i\nabla \bu_i, \nabla \bv_{i,0})_T- \langle \bv_{i,0} - \bv_{i,b} ,A_i \nabla \bu_i \cdot \bn_i \rangle_{\partial T} \right)\\
			&-\sum_{e \in \Gamma} \langle \bv_{i,b} , A_1 \nabla \bu_1 \cdot \bn_1 + A_2 \nabla \bu_2 \cdot \bn_2 \rangle_e,\\
		\end{split}
	\end{eqnarray}
	where we have used the fact that $\sum_{e \in \mathcal{E}_h^S} \langle \bv_b,\nabla \bu \cdot \bn\rangle_e=0$.
	
	Then by integration by parts and the fact that $\sum_{e \in \mathcal{E}_h^S} \langle \bv_b, p \bn \rangle_e =0$, we have
	\begin{align}
		\begin{split}
			(\nabla p, \bv_0)
			=&\sum_{i=1}^2 \sum_{T \in \mathcal{T}_h^i} (\nabla p_i, \bv_{i,0})_T\\
			=&\sum_{i=1}^2 \sum_{T \in \mathcal{T}_h^i}(-(\nabla \cdot \bv_{i,0},p_i)_T+ \langle \bv_{i,0}, p_i \bn \rangle_{\partial T}) \label{EEP5}\\
			=&\sum_{i=1}^2 \sum_{T \in \mathcal{T}_h^i} \left(-(\nabla \cdot \bv_{i,0}, p_i)_T +\langle \bv_{i,0}-\bv_{i,b}, p_i\bn \rangle_{\partial T} \right)+\sum_{e \in \Gamma} \langle \bv_{i,b} , p_1 \bn_1 +p_2 \bn_2 \rangle_e.
		\end{split}
	\end{align}
	Substituting Eq.(\ref{EEP4}) and Eq.(\ref{EEP5}) into Eq.(\ref{EEP3}) and using the interface condition (\ref{interface 2}), we have 
	\begin{eqnarray}\label{EEP6}
		\begin{split}
			&\sum_{T \in \mathcal{T}_h} \left((A \nabla \bu, \nabla \bv_0)_T -\langle \bv_0 -\bv_b, A \nabla \bu \cdot \bn  \rangle_{\partial T} -(\nabla \cdot \bv_0, p)_T+\langle \bv_0 -\bv_b, p \bn \rangle_{\partial T} \right)\\
			= &(\bbf,\bv_0)+\sum_{e \in \Gamma} \langle \bv_b, \bm{\psi} \rangle_e.
		\end{split}
	\end{eqnarray}
	Adding Eq.(\ref{EEP2}) to Eq.(\ref{EEP1}), in light of Eq.(\ref{EEP6}), the proof of Eq.(\ref{EEE}) is complete.
\end{proof}

\begin{lemma}\label{error equation}
	For the $(\bu_i, p_i)\in [H^1(\Omega_i)]^2 \times L^2(\Omega_i)$ with $i=1,2$ and $p \in L_0^2(\Omega)$ satisfying Eqs.(\ref{model 1})-(\ref{interface 2}), the errors $\be_h$ and $\varepsilon_h$ satisfy the following equations: 
	\begin{eqnarray}
		a_s(\be_h,\bv)+b(\bv,\varepsilon_h)&=&\ell_1(\bu,\bv)-\ell_2(p,\bv)+\ell_3(\bu,\bv)+s(Q_h \bu ,\bv),\label{error equation 1}\\
		b(\be_h,q)&=&-\ell_4(\bu,q), \label{error equation 2}
	\end{eqnarray}
	where
	\begin{eqnarray*}
		\ell_4(\bu,q)=\sum_{i=1}^2\sum_{e \in \Gamma }\langle Q_b \bu_i -\bu_i, q_i\bn_i \rangle_e,
	\end{eqnarray*}
	for any $\bv \in V_h^0$ and $q \in W_h$.
\end{lemma}
\begin{proof}
	Since the exact solution $(\bu, p)$ satisfies the Eqs.(\ref{model 1})-(\ref{interface 2}), according to Lemma \ref{EE}, we have
	\begin{align*}
		a(Q_h \bu, \bv)+b(\bv, \mathcal{Q}_h p)
		=&(\bbf,\bv_0)+\sum_{e \in \Gamma} \langle \bm{\psi}, \bv_b  \rangle_e+\ell_1(\bu,\bv)-\ell_2(p,\bv)+\ell_3(\bu,\bv).
	\end{align*}
	Adding $s(Q_h \bu, \bv)$ to both sides of the above equation and subtracting from Eq.(\ref{algorithm 1}) leads to Eq.(\ref{error equation 1}). By incorporating with Lemma \ref{weak divergence exchange}, we have
	\begin{eqnarray}\label{eep6}
		\begin{split}
			-(\nabla_w \cdot (Q_h \bu), q)=&\sum_{i=1}^2 \sum_{T \in \mathcal{T}_h^i} -(\nabla_w \cdot (Q_h \bu_i), q_i)_T\\
			=&\sum_{i=1}^2 \sum_{T \in \mathcal{T}_h} (\mathcal{Q}_h(\nabla \cdot \bu_i) ,q_i)_T-\sum_{i=1}^2 \sum_{e \in \Gamma }\langle Q_b \bu_i -\bu_i, q_i\bn_i \rangle_e\\
			=&\sum_{i=1}^2 \sum_{e \in \Gamma }-\langle Q_b \bu_i -\bu_i, q_i\bn_i \rangle_e.
		\end{split}
	\end{eqnarray}
	Then, subtracting Eq.(\ref{eep6}) from Eq.(\ref{algorithm 2}), we derive
	\begin{eqnarray}
		b(\be_h,q)=-\ell_4(\bu,q).
	\end{eqnarray}
	Hence the error equations are proved.
\end{proof}

\section{Error Estimate in the Energy Norm}
In this section, we establish optimal order estimates for error $\be_h$ of velocity function and error $\varepsilon_h$ of pressure function in the energy norm.
\begin{lemma} \cite{WGStokes1}
	For any $\bv=\{ \bv_0, \bv_b\} \in V_h$ and $T \in \mathcal{T}_h^S $, we obtain
	\begin{eqnarray}\label{estimate 1}
		\| \nabla \bv_0 \|^2_T \leqslant C \3bar \bv \3bar^2.
	\end{eqnarray}
\end{lemma}

\begin{lemma}\label{H1 estimates}
	Suppose that $\bu_i \in[H^{k+1}(\Omega_i)]^2$ and $p_i \in H^k(\Omega_i)$ with $i=1,2$, we have
	\begin{eqnarray}
		|\ell_1(\bu,\bv)|  & \leqslant & C h^k(\|\bu_1\|_{k+1, \Omega_1}+\| \bu_2 \|_{k+1, \Omega_2}) \3bar \bv \3bar,\label{H1 estimates 1}\\
		|\ell_2(p,\bv)| &\leqslant & C h^k(\| p_1 \|_{k, \Omega_1}+\| p_2 \|_{k, \Omega_2}) \3bar \bv \3bar,\label{H1 estimates 2}\\
		|\ell_3(\bu,\bv)| &\leqslant & C h^k(\|\bu_1\|_{k+1, \Omega_1}+\| \bu_2 \|_{k+1, \Omega_2}) \3bar \bv \3bar,\label{H1 estimates 3}\\
		|\ell_4(\bu,q)| &\leqslant & C h^k(\|\bu_1\|_{k+1, \Omega_1}+\| \bu_2 \|_{k+1, \Omega_2}) \| q\|, \label{H1 estimates 4}\\
		|s(Q_h \bu ,\bv)| &\leqslant & C h^k(\|\bu_1\|_{k+1, \Omega_1}+\| \bu_2 \|_{k+1, \Omega_2}) \3bar \bv \3bar,\label{H1 estimates 5}
	\end{eqnarray}
	for all $\bv \in  V_h^0$ and $q \in W_h$.
\end{lemma}

\begin{proof}
	As to the estimate (\ref{H1 estimates 1}), according to the Cauchy-Schwarz inequality, we have
	\begin{eqnarray}\label{HEP1}
		\begin{split}
			|\ell_1(\bu,\bv)|=&\left|
			\sum_{T \in \mathcal{T}_h} \langle \bv_0 -\bv_b, A (\nabla \bu)\cdot \bn -A \mathbb{Q}_h(\nabla \bu)\cdot \bn \rangle_{\partial T}
			\right|\\
			\leqslant& A \sum_{T \in \mathcal{T}_h} \| \bv_0 -\bv_b  \|_{\partial T} \| \nabla \bu -\mathbb{Q}_h(\nabla \bu) \|_{\partial T}\\
			\leqslant& A \left( \sum_{T \in \mathcal{T}_h}  \| \bv_0 -\bv_b  \|^2_{\partial T} \right)^{\frac{1}{2}} \left( \sum_{T \in \mathcal{T}_h} \| \nabla \bu - \mathbb{Q}_h(\nabla \bu) \|^2_{\partial T}  \right)^{\frac{1}{2}}.
		\end{split}
	\end{eqnarray}
	It follows from the trace inequality and the estimate (\ref{projectorestimate2}) that
	\begin{eqnarray}\label{HEP2}
		\begin{split}
			&  \sum_{T \in \mathcal{T}_h} \| \nabla \bu - \mathbb{Q}_h(\nabla \bu) \|^2_{\partial T}\\
			\leqslant&  \sum_{T \in \mathcal{T}_h} C \left( h_T^{-1} \| \nabla \bu - \mathbb{Q}_h(\nabla \bu) \|^2_T
			+ h_T \| \nabla(\nabla \bu - \mathbb{Q}_h(\nabla \bu))\|^2_T \right)\\
			\leqslant&  C h^{2k-1} (\| \bu_1 \|^2_{k+1,\Omega_1} +\| \bu_2\|^2_{k+1,\Omega_2} ).
		\end{split}
	\end{eqnarray}
	Using the triangle inequality, the trace inequality, the Cauchy-Schwarz inequality and the estimate (\ref{projectorestimate1}), as well as the estimate (\ref{estimate 1}), we obtain
	\begin{eqnarray}\label{HEP3}
		\begin{split}
			& \sum_{T \in \mathcal{T}_h} \| \bv_0 -\bv_b  \|^2_{\partial T}\\
			\leqslant&  \sum_{T \in \mathcal{T}_h^S} 2 \left( \|\bv_0 - Q_b\bv_0 \|^2_{\partial T}+ \|Q_b\bv_0-\bv_b \|^2_{\partial T} \right) + \sum_{T \in \mathcal{T}_h^I} \| \bv_0 -\bv_b  \|^2_{\partial T}\\
			\leqslant&  C\left(\sum_{T \in \mathcal{T}_h^S} h_T \| \nabla \bv_0 \|_T^2 \right)
			+C h \left(\sum_{T \in \mathcal{T}_h^S} h_T^{-1}\|Q_b\bv_0-\bv_b \|^2_{\partial T} \right) + h \3bar \bv \3bar^2  \\
			\leqslant&  C h \3bar \bv \3bar^2.
		\end{split}
	\end{eqnarray}
	By the estimates (\ref{HEP2})-(\ref{HEP3}), the proof of estimate (\ref{H1 estimates 1}) is complete.
	
	For the estimate (\ref{H1 estimates 2}), the same techniques of proving the estimate (\ref{H1 estimates 1}) can be applied to derive
	\begin{eqnarray}\label{HEP4}
		\begin{split}
			|\ell_2(p,\bv)|=&\left|
			\sum_{T \in \mathcal{T}_h} \langle \bv_0 -\bv_b,p\bn-(\mathcal{Q}_h p )\bn \rangle_{\partial T}
			\right|\\
			\leqslant & \sum_{T \in \mathcal{T}_h} \| \bv_0 -\bv_b\|_{\partial T} \| p- \mathcal{Q}_h p\|_{\partial T}\\
			\leqslant & \left( \sum_{T \in \mathcal{T}_h} h_T^{-1} \|\bv_0 -\bv_b \|_{\partial T}^2
			\right)^{\frac{1}{2}} \left(\sum_{T \in \mathcal{T}_h} h_T \|p -\mathcal{Q}_h p\|_{\partial T}^2
			\right)^{\frac{1}{2}}\\
			\leqslant & C h^k (\|p_1\|_{k,\Omega_1}+\|p_2\|_{k,\Omega_2})\3bar \bv \3bar.
		\end{split}
	\end{eqnarray}
	Similarly, we obtain
	\begin{eqnarray}\label{HEP5}
		\begin{split}
			|s(Q_h \bu,\bv)|=&\Big|
			\sum_{T \in \mathcal{T}_h}A h_T^{-1} \langle
			Q_b(Q_0 \bu) -Q_b \bu,Q_b \bv_0 -\bv_b \rangle_{\partial T \setminus (\partial T \cap \Gamma)} \\
			&+A h_T^{-1} \langle Q_0 \bu  - Q_b \bu, \bv_0 -\bv_b \rangle_{\partial T \cap \Gamma}
			\Big|\\
			\leqslant & \sum_{T \in \mathcal{T}_h} h_T^{-1} \Big(\|Q_0 \bu -\bu \|_{\partial T \setminus (\partial T \cap \Gamma)} \| Q_b \bv_0 -\bv_b\|_{\partial T \setminus (\partial T \cap \Gamma)} \\
			&+h_T^{-1} \|Q_0 \bu  - Q_b \bu \|_{\partial T \cap \Gamma} \|\bv_0 -\bv_b \|_{\partial T \cap \Gamma}  \Big)\\
			\leqslant & Ch^k (\| \bu_1 \|_{k+1,\Omega_1} +\| \bu_2\|_{k+1,\Omega_2}) \3bar \bv \3bar.
		\end{split}
	\end{eqnarray}
	For $\ell_3(\bu,\bv)$, by the definition of $\3bar \cdot \3bar$, we have
	\begin{eqnarray}
		\begin{split}
			|\ell_3(\bu,\bv)|=&\left|
			\sum_{i=1}^2 \sum_{e \in \Gamma} \langle 
			Q_b \bu_i -\bu_i , A_i \nabla_w \bv_i \cdot \bn_i \rangle_e \right|\\
			\leqslant & C\sum_{i=1}^2 \sum_{e \in \Gamma} \|Q_b \bu_i -\bu_i \|_{e} \|A_i^{\frac{1}{2}}\nabla_w \bv_i \|_{e}\\
			\leqslant & \sum_{i=1}^2 \left( \sum_{e \in \Gamma} \| Q_b \bu_i -\bu_i\|_e^2 \right)^{\frac{1}{2}} \left( \sum_{e \in \Gamma} \|A_i^{\frac{1}{2}} \nabla_w \bv_i \|_e^2 \right)^{\frac{1}{2}}\\
			\leqslant & Ch^k (\| \bu_1 \|_{k+1,\Omega_1} +\| \bu_2\|_{k+1,\Omega_2}) \3bar \bv \3bar.
		\end{split}
	\end{eqnarray}
	Similarly, we get
	\begin{eqnarray}
		\begin{split}
			|\ell_4(\bu,q)|=&\left|
			\sum_{i=1}^2 \sum_{e \in \Gamma} \langle 
			Q_b \bu_i -\bu_i , q_i \bn_i \rangle_e
			\right|\\
			\leqslant & C\sum_{i=1}^2 \sum_{e \in \Gamma} \|Q_b \bu_i -\bu_i \|_{e} \|q\|_{e}\\
			\leqslant & \sum_{i=1}^2 \left( \sum_{e \in \Gamma} \| Q_b \bu_i -\bu_i\|_e^2 \right)^{\frac{1}{2}} \left( \sum_{e \in \Gamma} \| q \|_e^2 \right)^{\frac{1}{2}}\\
			\leqslant & Ch^k (\| \bu_1\|_{k+1} +\| \bu_2\|_{k+1}) \|q\|.
		\end{split}
	\end{eqnarray}
	The proof of the lemma is complete.
\end{proof}

Based on error equations (\ref{error equation 1})-(\ref{error equation 2}) and estimates (\ref{H1 estimates 1})-(\ref{H1 estimates 5}), we give the following error estimate.
\begin{theorem}\label{H1ERROR}
	Assuming $(\bu_i ,p_i) \in [H^{k+1}(\Omega_i)]^2 \times H^k (\Omega_i)$ with $i=1,2$ are the exact solutions of the Eqs.(\ref{model 1})-(\ref{interface 2}) and $(\bu_h,p_h) \in V_h \times W_h$ are numerical solutions obtained from the WG scheme, then we have
	\begin{eqnarray}\label{H1error}
		\begin{split}
			\3bar Q_h \bu - \bu_h \3bar +\|\mathcal{Q}_h p -p_h\| 
			\leqslant  C h^k (\|\bu_1\|_{k+1,\Omega_1}+\|\bu_2\|_{k+1,\Omega_2}+\|p_1\|_{k,\Omega_1} +\|p_2\|_{k,\Omega_2}).
		\end{split}
	\end{eqnarray}
\end{theorem}
\begin{proof}
	Choosing $\bv= \be_h$ in Eq.(\ref{error equation 1}) and $q = \varepsilon_h$ in Eq.(\ref{error equation 2}) and adding the two equations, then we have
	\begin{eqnarray}\label{H1errorP1}
		a_s(\be_h,\be_h)=\ell_1(\bu,\be_h)-\ell_2(p,\be_h)+\ell_3(\bu,\be_h)+s(Q_h \bu, \be_h)+\ell_4(\bu,\varepsilon_h).
	\end{eqnarray}
	To simplify, let $\delta$ represent $\|\bu_1\|_{k+1,\Omega_1}+\|\bu_2\|_{k+1,\Omega_2}+\|p_1\|_{k,\Omega_1} +\|p_2\|_{k,\Omega_2}$. According to Eqs.(\ref{H1 estimates 1})-(\ref{H1 estimates 5}), we derive
	\begin{eqnarray}\label{H1errorP2}
		\begin{split}
			\3bar \be_h \3bar^2 \leqslant& C h^k \delta \3bar \be_h \3bar +C h^k(\|\bu_1\|_{k+1,\Omega_1}+\|\bu_2\|_{k+1,\Omega_2}) \|\varepsilon_h \|.
		\end{split}
	\end{eqnarray}
	According to Eq.(\ref{error equation 1}) and the boundedness of $a_s(\cdot, \cdot)$, we have
	\begin{eqnarray}\label{H1errorP3}
		\begin{split}
			b(\bv,\varepsilon_h)=&\ell_1(\bu,\bv)-\ell_2(p,\bv)-\ell_3(\bu,\bv)+s(Q_h \bu, \bv)-a_s(\be_h,\bv)\\
			\leqslant & C \3bar \be_h \3bar \3bar \bv \3bar 
			+Ch^k \delta \3bar \bv \3bar.
		\end{split}
	\end{eqnarray}
	In particular, we take $\bv =\{ \nabla \varepsilon_h,\mathbf{0}\} \in V_h^0$ to obtain 
	\begin{eqnarray}\label{H1errorP4}
		\begin{split}
			b(\bv,\varepsilon_h) =&\sum_{T \in \mathcal{T}_h} -(\nabla_w \cdot \bv,\varepsilon_h)_T \\
			=&\sum_{T \in \mathcal{T}_h} \left( (\bv_0 , \nabla \varepsilon_h )_T - \langle \bv_b, \varepsilon_h \bn \rangle_{\partial T}
			\right) \\
			=& \|\nabla \varepsilon_h\|^2 = |\varepsilon_h|_1^2.
		\end{split}
	\end{eqnarray}
	Moreover, for $\3bar \bv \3bar$, it follows from the Cauchy-Schwarz inequality, the trace inequality and the inverse inequality that
	\begin{eqnarray}\label{H1errorP5}
		\begin{split}
			\|\nabla_w \bv\|_T^2=&-(\bv_0,\nabla \cdot (\nabla_w \bv))_T +\langle \bv_b ,\nabla_w \bv \cdot \bn \rangle_{\partial T}\\
			=&(\nabla \bv_0 ,\nabla_w \bv)_T -\langle \bv_0 ,\nabla_w \bv  \cdot \bn \rangle_{\partial T }\\
			=&-(\bv_0,\nabla \cdot (\nabla \bv_0))_T -\langle \bv_0 ,\nabla_w \bv  \cdot \bn \rangle_{\partial T }\\
			=& (\nabla \bv_0 ,\nabla \bv_0)_T -\langle \bv_0, (\nabla_w \bv + \nabla \bv_0)\cdot \bn \rangle_{\partial T}\\
			\leqslant & \| \nabla \bv_0 \|_T^2+\|\bv_0\|_{\partial T } \|\nabla_w \bv + \nabla \bv_0 \|_{\partial T }\\
			\leqslant & Ch^{-2} \| \bv_0 \|_T^2 +Ch^{-1} \| \bv_0\|_T \| \nabla_w \bv \|_T\\
			\leqslant & Ch^{-2} \|  \bv_0 \|_T^2 + \frac{1}{2} \| \nabla_w \bv \|_T^2,
		\end{split}
	\end{eqnarray}
	i.e.
	\begin{eqnarray}\label{H1errorP6}
		\sum_{T \in \mathcal{T}_h} \| \nabla_w \bv\|_T^2 \leqslant Ch^{-2} \sum_{T \in \mathcal{T}_h} \|\bv_0 \|_T^2 \leqslant Ch^{-2}|\varepsilon_h|^2_1.
	\end{eqnarray}
	For the stabilizer term, we obtain
	\begin{eqnarray}\label{H1errorP7}
		\begin{split}
			& \sum_{T \in \mathcal{T}_h}A h_T^{-1} \|Q_b \bv_0 -\bv_b\|^2_{\partial T \setminus (\partial T \cap \Gamma)} +A h_T^{-1} \|\bv_0 -\bv_b\|^2_{\partial T \cap \Gamma} \\
			\leqslant & \sum_{T \in \mathcal{T}_h} Ch_T^{-1} \| \bv_0\|_{\partial T} \leqslant \sum_{T \in \mathcal{T}_h} Ch_T^{-2} \| \bv_0\|_T \leqslant Ch^{-2} |\varepsilon_h|_1^2.
		\end{split}
	\end{eqnarray}
	Adding (\ref{H1errorP6}) to (\ref{H1errorP7}), we get
	\begin{eqnarray}\label{H1errorP8}
		\3bar \bv \3bar \leqslant Ch^{-1} |\varepsilon_h|_1.
	\end{eqnarray}
	Substituting Eq.(\ref{H1errorP4}) and Eq.(\ref{H1errorP8}) into Eq.(\ref{H1errorP3}), we have
	\begin{eqnarray}\label{H1errorP9}
		|\varepsilon_h|_1 \leqslant Ch^{-1} \3bar \be_h \3bar +Ch^{k-1} \delta.
	\end{eqnarray}
	Next, according to the inf-sup condition (\ref{InfSup1}), we derive
	\begin{eqnarray}\label{H1errorP10}
		C_1 \|\varepsilon_h\| \leqslant C \3bar \be_h \3bar + C h^k \delta.
	\end{eqnarray}
	Substituting Eq.(\ref{H1errorP10}) into Eq.(\ref{H1errorP2}) gives rise to  
	\begin{eqnarray}\label{H1errorP11}
		\begin{split}
			\3bar \be_h \3bar^2
			&\leqslant Ch^k \delta  \3bar \be_h \3bar +Ch^k\left(\|\bu_1\|_{k+1,\Omega_1}+\|\bu_2\|_{k+1,\Omega_2}\right) \left( C\3bar \be_h \3bar +Ch^k  \delta \right)\\
			&\leqslant  C h^k \delta \3bar \be_h \3bar +C h^{2k} \delta^2\\
			&\leqslant C h^{2k} \delta^2 +\frac{1}{2} \3bar \be_h \3bar^2,
		\end{split}
	\end{eqnarray}
	therefore we have
	\begin{eqnarray}
		\3bar \be_h \3bar \leqslant C h^k \delta , \quad
		\| \varepsilon_h \| \leqslant C h^k \delta.
	\end{eqnarray}
	Therefore, the estimate (\ref{H1error}) holds true.
\end{proof}
\section{ Error Estimate in the $L^2$ Norm} In this section, we use the duality argument to derive an error estimate for $\be_0=Q_0 \bu - \bu_0$ in the $L^2$ norm. Now we consider the following problem: seeking $(\bw, \theta)$ to satisfy
\begin{eqnarray}
	- \nabla \cdot (A \nabla \bm{\omega})+\nabla \theta &=& \be_0 ~~~~ in ~\Omega, \label{dual problem 1} \\
	\nabla \cdot \bm{\omega} &=&0 ~~~~~ in ~ \Omega, \label{dual problem 2} \\
	\bm{\omega}&=& 0 ~~~~~ on ~\partial \Omega, \label{dual problem 3}
\end{eqnarray}
Assume the solution $(\bm{\omega}, \theta)$ of dual problem (\ref{dual problem 1})-(\ref{dual problem 3}) has $[H^2(\Omega)]^2 \times [H^1(\Omega)]$-regularity estimate, i.e.
\begin{eqnarray}
	\| \bm{\omega} \|_2 +\| \theta \|_1 \leqslant C \|\be_0 \|.
\end{eqnarray}
\begin{theorem}
	Based on the assumptions in Theorem \ref{H1ERROR}, we obtain the following error estimate:
	\begin{eqnarray}
		\| Q_0 \bu -\bu_0\| \leqslant C h^{k+1}(\| \bu_1 \|_{k+1,\Omega_1}+\| \bu_2\|_{k+1,\Omega_2}+\|p_1\|_{k,\Omega_1}+\|p_2\|_{k,\Omega_2}).
	\end{eqnarray}
\end{theorem}
\begin{proof}
	Due to $(\bw,\theta)$ satisfies the Eq.(\ref{dual problem 1}) with $\bbf = \be_0=Q_0 \bu -\bu_0$, then choosing $\bv = \be_h $ in Eq.(\ref{EEE}) and $q = \mathcal{Q}_h \theta $ in Eq.(\ref{error equation 2}) leads to
	\begin{align}
		a_s(Q_h \bm{\omega} ,\be_h)+b(\be_h, \mathcal{Q}_h \theta)&=(\be_0,\be_0)+\ell_1(\bm{\omega},\be_h)
		-\ell_2(\theta,\be_h)+\ell_3(\bm{\omega},\be_h)+s(Q_h \bm{\omega},\be_h),\\
		b(\be_h, \mathcal{Q}_h \theta)&=-\ell_4(\bu, \mathcal{Q}_h \theta).
	\end{align}
	According to the definition of the weak divergence operator and Lemma \ref{weak divergence exchange}, we have
	\begin{eqnarray}
		\begin{split}
			b(Q_h\bm{\omega},\varepsilon_h)=&-\sum_{i=1}^2 \sum_{T \in \mathcal{T}_h} (\nabla_w \cdot Q_h\bm{\omega}_i,\varepsilon_h)_T\\
			=&-\sum_{i=1}^2 \sum_{T \in \mathcal{T}_h} (\mathcal{Q}_h(\nabla \cdot \bm{\omega}_i),\varepsilon_h)_T -\sum_{i=1}^2 \sum_{e \in \Gamma} \langle Q_b \bm{\omega}_i -\bm{\omega}_i,\varepsilon_h \bn_i \rangle_e\\
			=&-\ell_4(\bm{\omega},\varepsilon_h).
		\end{split}
	\end{eqnarray}
	Therefore, we get
	\begin{eqnarray}
		\begin{split}
			a_s(Q_h\bm{\omega},\be_h)+b(Q_h\bm{\omega},\varepsilon_h)=&\|\be_0\|^2+\ell_1(\bm{\omega},\be_h)-\ell_2(\theta,\be_h)+\ell_3(\bm{\omega},\be_h)\\
			&+s(Q_h\bm{\omega},\be_h)-\ell_4(\bm{\omega},\varepsilon_h)+\ell_4(\bu,\mathcal{Q}_h \theta).
		\end{split}
	\end{eqnarray}
	From the above equations and Eq.(\ref{error equation 1}), we obtain
	\begin{align*}
		\begin{split}
			\|\be_0\|^2=&\ell_1(\bu,Q_h\bm{\omega})-\ell_2(p,Q_h\bm{\omega})+\ell_3(\bu,Q_h\bm{\omega})+s(Q_h \bu,Q_h\bm{\omega})-\ell_4(\bu,\mathcal{Q}_h\theta)\\
			&-\ell_1(\bm{\omega},\be_h)+\ell_2(\theta,\be_h)-\ell_3(\bm{\omega},\be_h)-s(Q_h\bm{\omega},\be_h)-\ell_4(\bm{\omega},\varepsilon_h).
		\end{split}
	\end{align*}
	According to Lemma \ref{H1 estimates}, we have
	\begin{align*}
		&|-\ell_1(\bm{\omega},\be_h)+\ell_2(\theta,\be_h)-\ell_3(\bm{\omega},\be_h)-s(Q_h\bm{\omega},\be_h-\ell_4(\bm{\omega},\varepsilon_h)|\\
		\leqslant & Ch(\| \bm{\omega}\|_2+\| \theta \|_1)(\3bar \be_h \3bar +\|\varepsilon_h \|) \\
		\leqslant &Ch\3bar \be_h \3bar \|\be_0\|.
	\end{align*}
	Each of the remaining terms is handled as follows.\\
	
	(1) For $\ell_1(\bu,Q_h \bm{\omega})$, we use the Cauchy-Schwarz inequality, the trace inequality and the estimate (\ref{projectorestimate1}) to derive
	\begin{align*}
		&\sum_{T \in \mathcal{T}_h} \langle Q_0 \bm{\omega} -\bm{\omega}, A \nabla \bu \cdot \bn -A \mathbb{Q}_h(\nabla \bu) \cdot \bn \rangle_{\partial T}\\
		\leqslant  & C\left(\sum_{T \in \mathcal{T}_h} \| Q_0 \bm{\omega} -\bm{\omega}\|^2_{\partial T} \right)^{\frac{1}{2}}
		\left(\sum_{T \in \mathcal{T}_h} \| \nabla \bu - \mathbb{Q}_h(\nabla \bu) \|^2_{\partial T} \right)^{\frac{1}{2}}\\
		\leqslant & \left( C \sum_{T \in \mathcal{T}_h} (h_T^{-1} \| Q_0 \bm{\omega}-\bm{\omega} \|_T^2
		+ h_T \| \nabla(Q_0 \bm{\omega} -\bm{\omega})\|_T^2 )\right)^{\frac{1}{2}}\\
		&\left( C \sum_{T \in \mathcal{T}_h} (h_T^{-1} \| \nabla \bu - \mathbb{Q}_h(\nabla \bu) \|_T^2
		+h_T \| \nabla(\nabla \bu - \mathbb{Q}_h(\nabla \bu))\|_T^2 )\right)^{\frac{1}{2}}\\
		\leqslant & \left( \sum_{T \in \mathcal{T}_h} C h_T^3 \|\bm{\omega}\|_2^2 \right)^{\frac{1}{2}}
		\left( \sum_{T \in \mathcal{T}_h} C h_T^{2k-1} \| \bu \|_{k+1}^2 \right)^{\frac{1}{2}}\\
		\leqslant & Ch^{k+1} \|\bm{\omega}\|_2 (\|\bu_1\|_{k+1,\Omega_1}+\|\bu_2\|_{k+1,\Omega_2}).
	\end{align*}
	Next, we use same techniques and the fact that $\sum_{e \in \mathcal{E}_h^S} \langle \bm{\omega}-Q_b \bm{\omega},A (\nabla \bu )\cdot \bn-A \mathbb{Q}_h(\nabla \bu) \cdot \bn \rangle_e=0$  to obtain
	\begin{align*}
		&\Big|\sum_{T \in \mathcal{T}_h} \langle \bm{\omega}-Q_b \bm{\omega}, A\nabla \bu \cdot \bn -A \mathbb{Q}_h(\nabla \bu) \cdot \bn \rangle_{\partial T}\Big|\\
		=&\Big|\sum_{e \in \mathcal{E}_h^I} \langle \bm{\omega} -Q_b \bm{\omega}, A (\nabla \bu ) \cdot \bn -A \mathbb{Q}_h(\nabla \bu) \cdot \bn \rangle_e\Big|\\
		\leqslant & \sum_{e \in \mathcal{E}_h^I }\| \bm{\omega} -Q_b \bm{\omega}\|_e \| A \nabla \bu  -A \mathbb{Q}_h(\nabla \bu) \|_e\\
		\leqslant & C \left( \sum_{e \in \mathcal{E}_h^I } \| \bm{\omega} -Q_b \bm{\omega}\|_e^2 \right)^{\frac{1}{2}}
		\left( \sum_{e \in \mathcal{E}_h^I} \| \nabla \bu  -\mathbb{Q}_h(\nabla \bu) \|_e^2  \right)^{\frac{1}{2}}\\
		\leqslant &C \left( h^{k-\frac{1}{2}} \| \bu \|_{k+1} \right)\left( \sum_{e \in \mathcal{E}_h^I} \| \bm{\omega}-Q_b \bm{\omega} \|_e^2 \right)^{\frac{1}{2}} ,
	\end{align*}
	where
	\begin{align*}
		\| \bm{\omega} -Q_b \bm{\omega}\|_e^2 \leqslant \| \bm{\omega} - Q_0 \bm{\omega}\|_e^2 \leqslant \|\bm{\omega} - Q_0 \bm{\omega} \|_{\partial T}^2 \leqslant C h_T^3 \|\bm{\omega}\|_{2}^2.
	\end{align*}
	Using the above inequality, we have
	\begin{align*}
		\sum_{T \in \mathcal{T}_h} \langle \bm{\omega}-Q_b \bm{\omega}, A \nabla \bu \cdot \bn -A \mathbb{Q}_h(\nabla \bu) \cdot \bn \rangle_{\partial T}
		\leqslant Ch^{k+1}\|\bm{\omega}\|_2(\|\bu_1\|_{k+1,\Omega_1}+\|\bu_2\|_{k+1,\Omega_2}).
	\end{align*}
	Therefore,
	\begin{eqnarray}\label{L2error1}
		\begin{split}
			|\ell_1(\bu,Q_h \bm{\omega})|=&\Big|\sum_{T \in \mathcal{T}_h} \langle Q_0 \bm{\omega} -Q_b \bm{\omega}, A \nabla \bu \cdot \bn -A \mathbb{Q}_h(\nabla \bu) \cdot \bn \rangle_{\partial T}\Big|\\
			\leqslant&\Big|\sum_{T \in \mathcal{T}_h} \langle Q_0 \bm{\omega} -\bm{\omega}, A \nabla \bu \cdot \bn -A \mathbb{Q}_h(\nabla \bu) \cdot \bn \rangle_{\partial T}\Big|\\
			&+\Big|\sum_{T \in \mathcal{T}_h} \langle \bm{\omega} -Q_b \bm{\omega} , A \nabla \bu \cdot \bn -A \mathbb{Q}_h(\nabla \bu) \cdot \bn \rangle_{\partial T}\Big|\\
			\leqslant & Ch^{k+1} \|\bm{\omega}\|_2 (\|\bu_1\|_{k+1,\Omega_1}+\|\bu_2\|_{k+1,\Omega_2}).
		\end{split}
	\end{eqnarray}
	
	(2)For $\ell_2(p,Q_h \bm{\omega})$, we use the same method as the proof of $\ell_1(\bu,Q_h \bm{\omega})$ to get 
	\begin{align*}
		&\sum_{T \in \mathcal{T}_h} \langle Q_0 \bm{\omega} -\bm{\omega}, (p-\mathcal{Q}_h p)\bn \rangle_{\partial T}\\
		\leqslant & \left( \sum_{T \in \mathcal{T}_h} \| Q_0 \bm{\omega} -\bm{\omega} \|_{\partial T}^2 \right)^{\frac{1}{2}}
		\left( \sum_{T \in \mathcal{T}_h} \| p-\mathcal{Q}_h p \|_{\partial T}^2\right)^{\frac{1}{2}}\\
		\leqslant & \left( C \sum_{T \in \mathcal{T}_h} (h_T^{-1} \| Q_0 \bm{\omega} -\bm{\omega} \|_T^2
		+ h_T \| \nabla(Q_0 \bm{\omega} -\bm{\omega})\|_T^2 )\right)^{\frac{1}{2}}\\
		&\left( C \sum_{T \in \mathcal{T}_h} (h_T^{-1} \| p-\mathcal{Q}_h p \|_T^2
		+ h_T \| \nabla(p-\mathcal{Q}_h p)\|_T^2 ) \right)^{\frac{1}{2}}\\
		\leqslant &\left( \sum_{T \in \mathcal{T}_h} C h_T^3 \| \bm{\omega} \|_2^2 \right)^{\frac{1}{2}} \left( \sum_{T \in \mathcal{T}_h} Ch^{2k-1}
		\| p\|_k^2  \right)^{\frac{1}{2}}\\
		\leqslant & C h^{k+1} \| \bm{\omega} \|_2 \| p \|_k.
	\end{align*}
	Similarly, we have
	\begin{align*}
		&\sum_{T \in \mathcal{T}_h} \langle \bm{\omega} - Q_b\bm{\omega}, (p-\mathcal{Q}_h p)\bn \rangle_{\partial T}
		\leqslant C h^{k+1} \| \bm{\omega}\|_2  (\|p_1\|_{k,\Omega_1}+\|p_2\|_{k,\Omega_2}).
	\end{align*}
	Therefore, for $\ell_2(p,Q_h \bw)$, we derive the following estimate
	\begin{eqnarray}\label{L2error2}
		|\ell_2(p,Q_h \bm{\omega})| \leqslant C h^{k+1} \| \bm{\omega} \|_2 \| p \|_k.
	\end{eqnarray}
	
	(3)For $s(Q_h \bu, Q_h \bm{\omega})$, we get
	\begin{eqnarray}\label{L2error3}
		\begin{split}
			&|s(Q_h \bu, Q_h \bm{\omega})|\\
			\leqslant& \left| \sum_{T \in \mathcal{T}_h}A h_T^{-1} \langle Q_b(Q_0 \bu)-Q_b \bu, Q_b(Q_0 \bm{\omega})-Q_b \bm{\omega} \rangle_{\partial T \setminus (\partial T \cap \Gamma)}\right|\\
			&+\left|\sum_{T \in \mathcal{T}_h}A h_T^{-1}  \langle Q_0 \bu-Q_b \bu, Q_0 \bm{\omega}-Q_b \bm{\omega} \rangle_{\partial T \cap \Gamma}\right|\\
			\leqslant &C  \left(\sum_{T \in \mathcal{T}_h} h_T^{-1} \|Q_0 \bu -\bu \|_{\partial T \setminus (\partial T \cap \Gamma)}^2 \right)^{\frac{1}{2}} \left(\sum_{T \in \mathcal{T}_h} h_T^{-1}\|Q_0 \bm{\omega} -\bm{\omega} \|_{\partial T \setminus (\partial T \cap \Gamma)}^2\right)^{\frac{1}{2}}\\
			&+ \left( \sum_{T \in \mathcal{T}_h} h_T^{-1}
			\|Q_0 \bu -\bu \|_{\partial T \cap \Gamma}^2+\|Q_b \bu -\bu \|_{\partial T \cap \Gamma}^2 \right)^{\frac{1}{2}}\\
			&\left( \sum_{T \in \mathcal{T}_h} h_T^{-1} \|Q_0 \bm{\omega}-\bm{\omega} \|_{\partial T \cap \Gamma}^2+\|Q_b \bm{\omega} -\bm{\omega} \|_{\partial T \cap \Gamma}^2\right)^{\frac{1}{2}}\\
			\leqslant & Ch^{k+1}(\|\bu_1\|_{k+1,\Omega_1}+\|\bu_2\|_{k+1,\Omega_2}) \|\bm{\omega}\|_2.
		\end{split}
	\end{eqnarray}
	
	(4)For $\ell_3(\bu,Q_h \bm{\omega})$, using the fact that $\sum_{i=1}^2 \sum_{e \in \Gamma} \langle Q_b \bu_i -\bu_i,Q_b(\nabla \bm{\omega}_i\cdot \bn_i) \rangle_e=0$, the Cauchy-Schwarz inequality, the trace inequality and the inverse inequality, we obtain
	\begin{align*}
		&|\ell_3(\bu,Q_h \bm{\omega})|\\
		=&\Big|\sum_{i=1}^2\sum_{e \in \Gamma} \langle Q_b \bu_i-\bu_i,A_i \nabla_w(Q_h \bm{\omega}_i)\cdot \bn \rangle_e\Big|\\
		\leqslant & C \sum_{i=1}^2 \sum_{e \in \Gamma} \|Q_b \bu_i-\bu_i \|_e \|\nabla_w(Q_h \bm{\omega}_i) \cdot \bn_i-\mathbb{Q}_h(\nabla \bm{\omega}_i) \cdot \bn_i \|_e\\
		&+C \sum_{i=1}^2 \sum_{e \in \Gamma} \|Q_b \bu_i-\bu_i \|_e \|\mathbb{Q}_h(\nabla \bm{\omega}_i)\cdot \bn_i-\nabla \bm{\omega}_i\cdot \bn_i \|_e\\
		&+C \sum_{i=1}^2 \sum_{e \in \Gamma} \|Q_b \bu_i-\bu_i \|_e \|Q_b(\nabla \bm{\omega}_i \cdot \bn_i)-(\nabla \bm{\omega}_i\cdot \bn_i) \|_e\\
		\leqslant & \left(\sum_{i=1}^2 \sum_{e \in \Gamma} \|Q_b \bu_i - \bu_i \|_e^2 \right)^{\frac{1}{2}} \left(\sum_{i=1}^2 \sum_{T \in \mathcal{T}_h^i} h_T^{-1}\|\nabla_w(Q_h \bm{\omega}_i)-\mathbb{Q}_h(\nabla \bm{\omega}_i) \|_T^2 \right)^{\frac{1}{2}} \\
		&+\left(\sum_{i=1}^2 \sum_{e \in \Gamma} \|Q_b \bu_i - \bu_i \|_e^2 \right)^{\frac{1}{2}} \left(\sum_{i=1}^2 \sum_{e \in \Gamma} \|\mathbb{Q}_h(\nabla \bm{\omega}_i)-\nabla \bm{\omega}_i \|_e^2 \right)^{\frac{1}{2}} \\
		&+\left(\sum_{i=1}^2 \sum_{e \in \Gamma} \|Q_b \bu_i - \bu_i \|_e^2 \right)^{\frac{1}{2}} \left(\sum_{i=1}^2 \sum_{e \in \Gamma} \|Q_b(\nabla \bm{\omega}_i \cdot \bn_i)-(\nabla \bm{\omega}_i\cdot \bn_i) \|_e^2 \right)^{\frac{1}{2}},
	\end{align*}
	where by the trace inequality and the estimate (\ref{projectorestimate1}), we have
	\begin{align*}
		&\sum_{i=1}^2 \sum_{e \in \Gamma} \|Q_b \bu_i - \bu_i \|_e^2 \\
		\leqslant &\sum_{i=1}^2 \sum_{e \in \Gamma} \|Q_0 \bu_i - \bu_i \|_e^2 \\
		\leqslant & \sum_{i=1}^2 \sum_{T \in \mathcal{T}_h^I} C \left(h_T^{-1} \|Q_0 \bu_i - \bu_i  \|_T^2 +h_T \|\nabla(Q_0 \bu_i - \bu_i) \|_T^2 \right) \\
		\leqslant & Ch^{2k+1} (\|\bu_1\|^2_{k+1,\Omega_1}+\|\bu_2\|^2_{k+1,\Omega_2}).
	\end{align*}
	Similarly, we get the following estimates:
	\begin{eqnarray*}
		\sum_{i=1}^2 \sum_{e \in \Gamma} \| \mathbb{Q}_h(\nabla \bm{\omega}_i)-\nabla \bm{\omega}_i\|_e^2 \leqslant Ch\|\bm{\omega}\|^2_2,\\
		\sum_{i=1}^2 \sum_{e \in \Gamma} \|Q_b(\nabla \bm{\omega}_i \cdot \bn)-(\nabla \bm{\omega}_i\cdot \bn) \|_e^2 \leqslant Ch\|\bm{\omega}\|^2_2.
	\end{eqnarray*}
	Next according to Eq.(\ref{weak gradient exchange 2}), we take $\tau =\nabla_w (Q_h \bm{\omega})-\mathbb{Q}_h(\nabla \bm{\omega})$ to lead to
	\begin{eqnarray*}
		\begin{split}
			&\|\nabla_w (Q_h \bm{\omega})-\mathbb{Q}_h(\nabla \bm{\omega})\|^2_T\\
			\leqslant& \|Q_b \bm{\omega}_i -\bm{\omega}_i\|_{\partial T}	\|\nabla_w (Q_h \bm{\omega})-\mathbb{Q}_h(\nabla \bm{\omega})\|^2_{\partial T} \\
			\leqslant& Ch^{-1} \|Q_b \bm{\omega}_i -\bm{\omega}_i\|_{\partial T} \|\nabla_w (Q_h \bm{\omega})-\mathbb{Q}_h(\nabla \bm{\omega})\|^2_T,
		\end{split}
	\end{eqnarray*}
	therefore we have
	\begin{eqnarray}
		\|\nabla_w (Q_h \bm{\omega})-\mathbb{Q}_h(\nabla \bm{\omega})\|_T \leqslant Ch^{-1} \|Q_b \bm{\omega}_i -\bm{\omega}_i\|_{\partial T} \leqslant Ch^{\frac{1}{2}}\|\bm{\omega} \|_2.
	\end{eqnarray}
	Combining the above four estimates, we get
	\begin{eqnarray}\label{L2error4}
		|\ell_3(\bu,Q_h \bm{\omega})| \leqslant Ch^{k+1} \|\bm{\omega}\|_2 (\|\bu_1\|_{k+1,\Omega_1}+\|\bu_2\|_{k+1,\Omega_2}).
	\end{eqnarray}
	
	(5)For $\ell_4(\bu,\mathcal{Q}_h \theta)$, by the fact that $\sum_{i=1}^2 \sum_{e \in \Gamma} \langle Q_b \bu_i -\bu_i, \mathcal{Q}_h (\theta \bn) \rangle_e =0$, we have
	\begin{eqnarray}\label{L2error5}
		\begin{split}
			|\ell_4(\bu,\mathcal{Q}_h \theta)| =&\Big|\sum_{i=1}^2 \sum_{e \in \Gamma} \langle Q_b \bu_i -\bu_i ,(\mathcal{Q}_h \theta)\bn \rangle_e\Big|\\
			\leqslant & \Big|\sum_{i=1}^2 \sum_{e \in \Gamma} \langle Q_b \bu_i -\bu_i, (\mathcal{Q}_h \theta)\bn -\theta \bn \rangle_e\Big|\\
			&+\Big|\sum_{i=1}^2 \sum_{e \in \Gamma} \langle Q_b \bu_i -\bu_i, \mathcal{Q}_h (\theta\bn) -\theta \bn \rangle_e\Big|\\
			\leqslant & Ch^{k+1} (\|\bu_1\|_{k+1,\Omega_1}+\|\bu_2\|_{k+1,\Omega_2})\|\theta\|_1.
		\end{split}
	\end{eqnarray}
	Combining the five estimates (\ref{L2error1})-(\ref{L2error5}), we obtain 
	\begin{align*}
		\|\be_0\|^2
		\leqslant& C h^{k+1}(\| \bu_1 \|_{k+1,\Omega_1}+\|\bu_2\|_{k+1,\Omega_2}+\|p_1\|_{k,\Omega_1}+\|p_2\|_{k,\Omega_2}) (\| \bw \|_2+\|\theta\|_1)\\
		&+Ch\3bar \be_h \3bar \|\be_0\|\\
		\leqslant &C h^{k+1}(\|\bu_1\|_{k+1,\Omega_1}+\|\bu_2\|_{k+1,\Omega_2}+\|p_1\|_{k,\Omega_1}+\|p_2\|_{k,\Omega_2})\| \be_0\|\\
		&+Ch\3bar \be_h \3bar \|\be_0\|.
	\end{align*}
	Finally, it follows from Theorem \ref{H1ERROR} that
	\begin{align*}
		\| \be_0 \|
		\leqslant & C h^{k+1}(\|\bu_1\|_{k+1,\Omega_1}+\|\bu_2\|_{k+1,\Omega_2}+\|p_1\|_{k,\Omega_1}+\|p_2\|_{k,\Omega_2})
		+Ch\3bar \be_h \3bar \\
		\leqslant & C h^{k+1}(\|\bu_1\|_{k+1,\Omega_1}+\|\bu_2\|_{k+1,\Omega_2}+\|p_1\|_{k,\Omega_1}+\|p_2\|_{k,\Omega_2}).
	\end{align*}
	The proof of theorem is complete.
\end{proof}
\section{Numerical Results} In this section, we give some numerical examples to validate the efficiency of the proposed WG method. We solve the interface problems in the domain $\Omega = [-1, 1] \times [-1, 1] $ with different interfaces.

\begin{example}\label{example3}
	In this example, we solve the interface problems with discontinuous velocity function and pressure function. The viscosity coefficient $A$ is continuous in $\Omega$. And the interface is described as
	\begin{align*}
		x^2+y^2=\frac{1}{4}.
	\end{align*}
	The exact solutions are
	\begin{eqnarray*}
		\bu_1&=&\left(\begin{array}{ccc}
			2 \sin y \cos y \cos x \\
			(\sin^2 y-2)\sin x
		\end{array}\right),\\
		\bu_2&=&\left(\begin{array}{ccc}
			- \cos(\pi x) \sin(\pi y)\\
			\sin(\pi x) \cos(\pi y)  
		\end{array}\right).\\
		p&=&\left \{\begin{array}{rcl}
			1 & \mbox{in} & \Omega_1 \\
			\frac{\pi}{16-\pi} & \mbox{in} & \Omega_2
		\end{array}\right.,\quad
		A=\left \{\begin{array}{rcl}
			1 & \mbox{in} & \Omega_1 \\
			1 & \mbox{in} & \Omega_2
		\end{array}\right..
	\end{eqnarray*}
\end{example}

\begin{figure}[htbp]
	\centering
	\begin{minipage}[t]{1\textwidth}
		\centering
		\includegraphics[width=6cm]{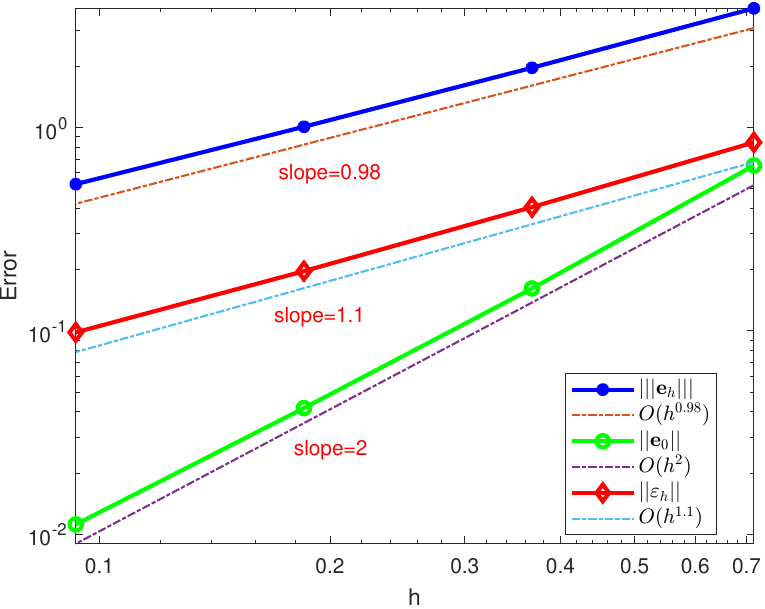}\hspace{0.5cm}
		\includegraphics[width=6cm]{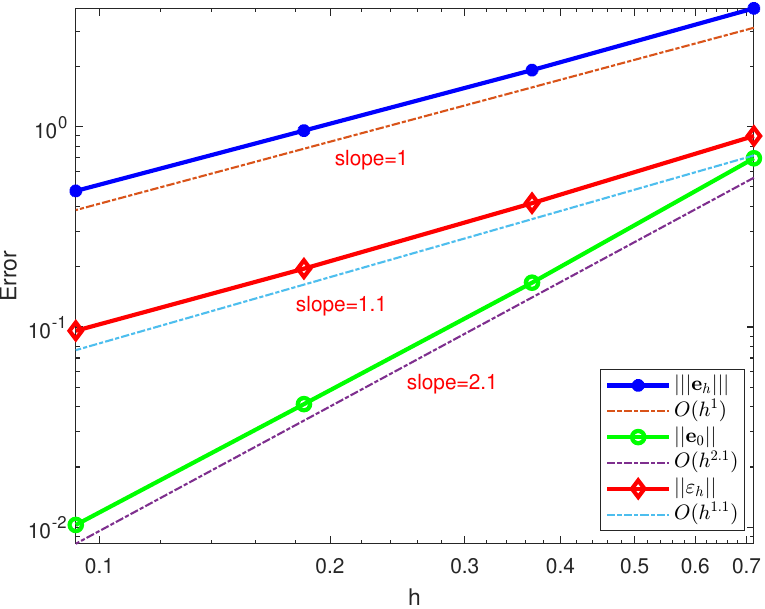}
		\caption{\small{ The numerical results for Example \ref{example3} on the curved triangular meshes (left) and straight triangular meshes (right) with $k=1$.}}\label{exam1-tri-k1}
	\end{minipage}
\end{figure}
\begin{figure}[htbp]
	\centering
	\begin{minipage}[t]{1\textwidth}
		\centering
		\includegraphics[width=6cm]{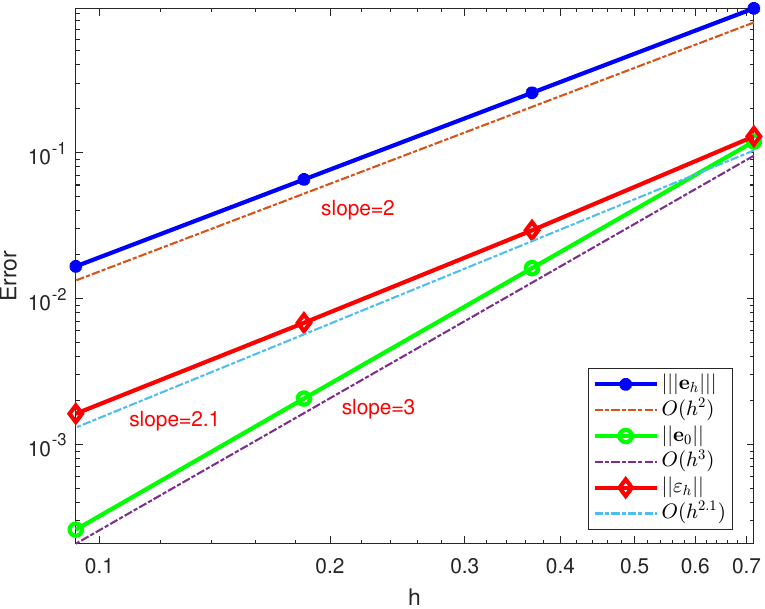}\hspace{0.5cm}
		\includegraphics[width=6cm]{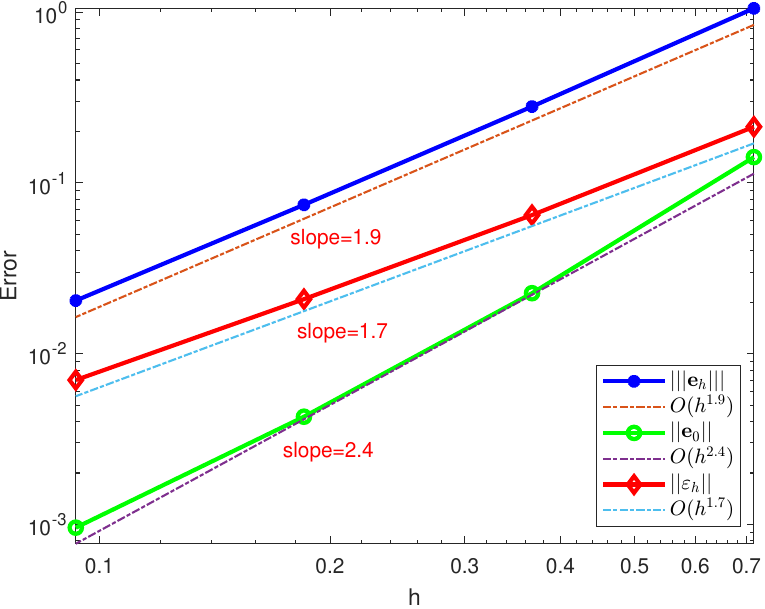}
		\caption{\small{ The numerical results for Example \ref{example3} on the curved triangular meshes (left) and straight triangular meshes (right) with $k=2$.}}\label{exam1-tri-k2}
	\end{minipage}
\end{figure}

\begin{figure}[htbp]
	\centering
	\begin{minipage}[t]{1\textwidth}
		\centering
		\includegraphics[width=6cm]{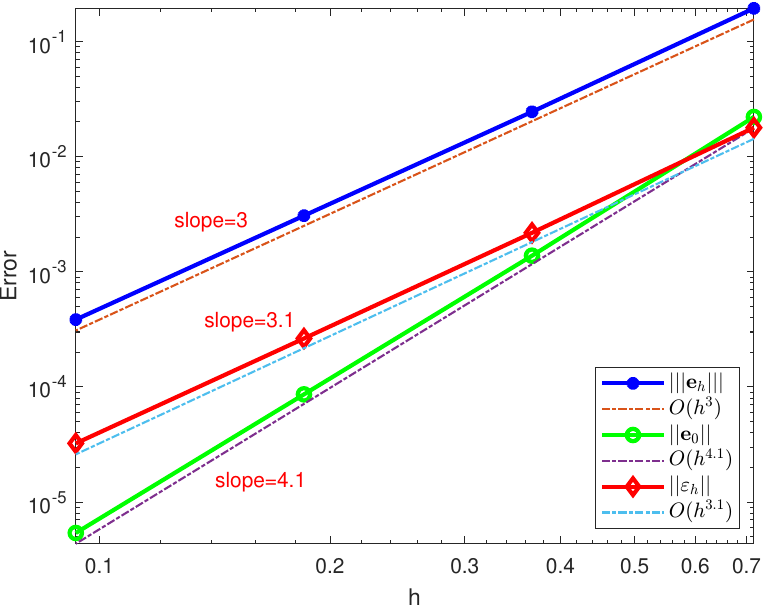}\hspace{0.5cm}
		\includegraphics[width=6cm]{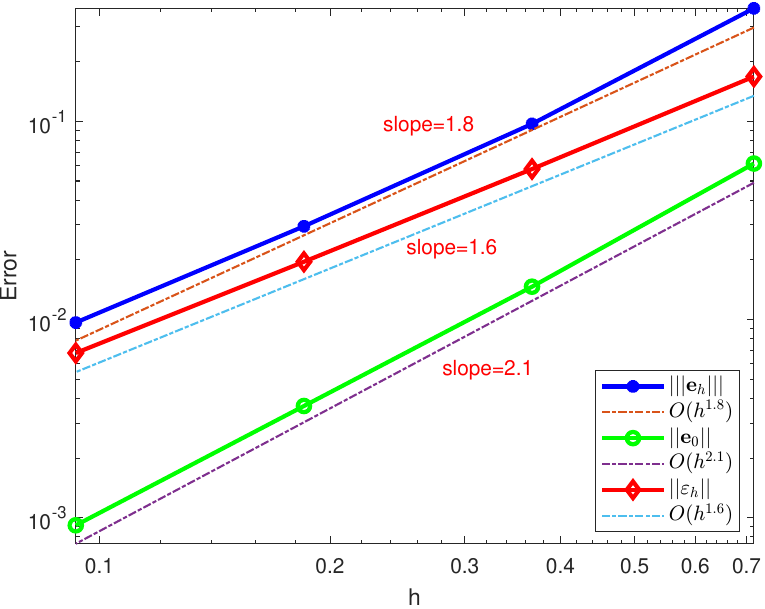}
		\caption{\small{ The numerical results for Example \ref{example3} on the curved triangular meshes (left) and straight triangular meshes (right) with $k=3$.}}\label{exam1-tri-k3}
	\end{minipage}
\end{figure}

\begin{figure}[t!]\label{test problem 1}
	\centering
	\begin{minipage}[t]{0.33\linewidth}
		\centering
		\includegraphics[width=1.1\linewidth]{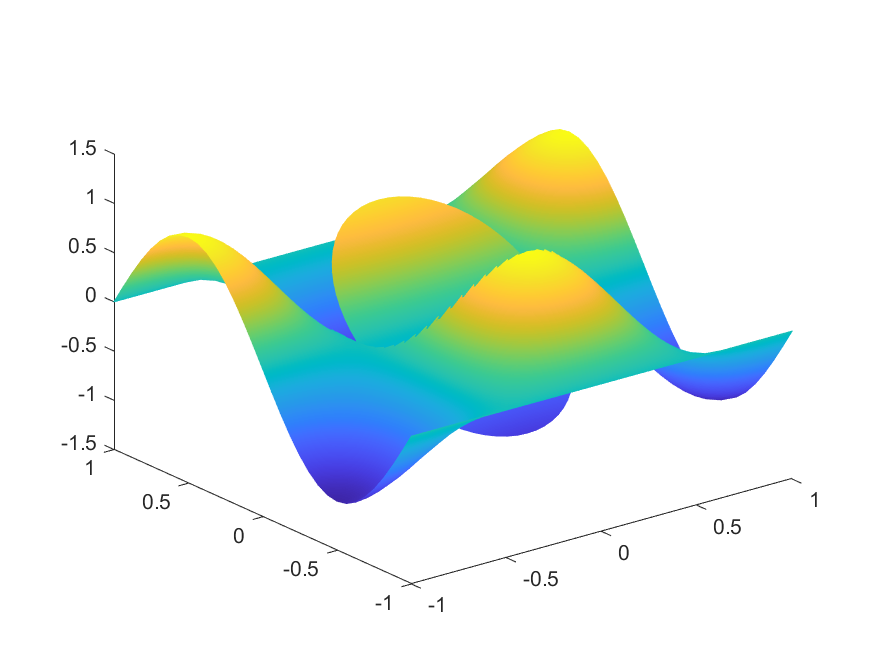}
	\end{minipage}%
	\begin{minipage}[t]{0.33\linewidth}
		\centering
		\includegraphics[width=1.1\linewidth]{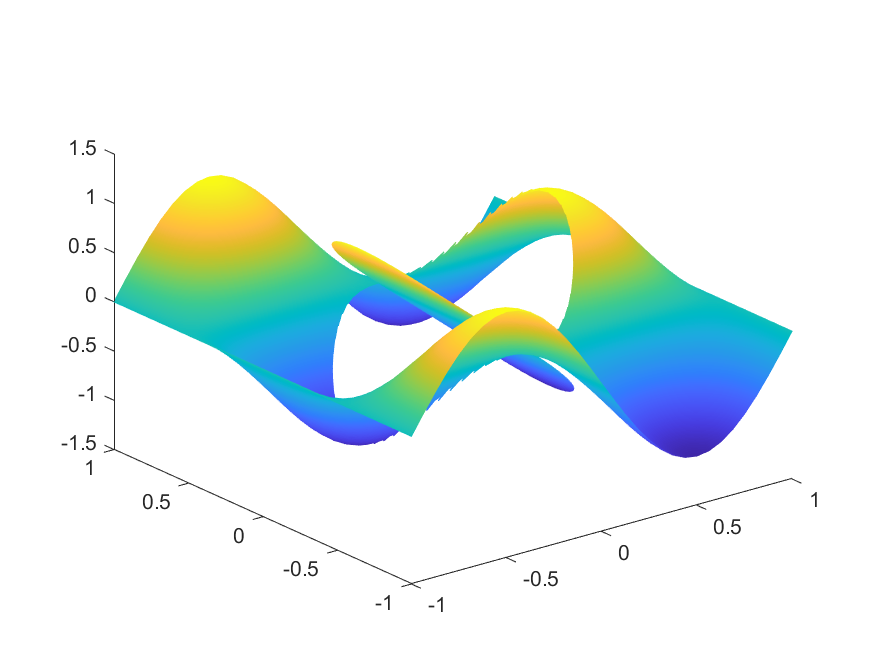}
	\end{minipage}
	\begin{minipage}[t]{0.33\linewidth}
		\centering \includegraphics[width=1.1\linewidth]{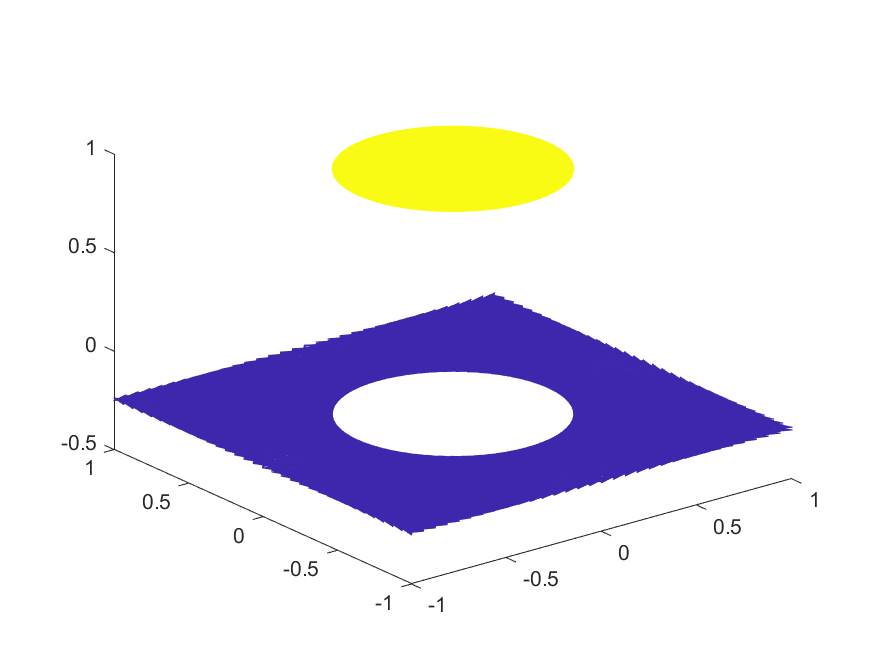}
	\end{minipage}
	\caption{The numerical solutions by $P_2$ WG element on curved triangular meshes level 3 in Example \ref{example3}. Left: The first component of $\bu_h$. Middle: The second component of $\bu_h$. Right: Pressure function $p_h$.}
	\label{figure7.1}
\end{figure} 


In Figures \ref{exam1-tri-k1} - \ref{exam1-tri-k3}, we compare the numerical results on the straight triangular meshes and curved triangular meshes. On the straight triangular meshes, we use the straight segments to replace the curved interface. As we can see, the optimal order convergence is obtained by the $P_1$ WG element in two cases. When using $P_2$ and $P_3$ WG elements to solve the problems, the orders of convergence are less than the optimal orders on the straight triangular meshes. However, all numerical solutions converge at the optimal rates on the curved triangular meshes. This comparison shows the advantages of our proposed WG scheme. The numerical solutions on the curved triangular meshes are plotted in Figure \ref{figure7.1}.

\begin{example}\label{example2}
	In the example, the interface between two subdomains is described as:
	\begin{align*}
		r=\frac{1}{2}+\frac{sin(2 \theta)}{4}.
	\end{align*}
	The exact solutions are
	\begin{eqnarray*}
	\bu_1&=&\left(\begin{array}{ccc}
		2 \sin y \cos y \cos x \\
		(\sin^2 y-2)\sin x
	\end{array}\right),\\
	\bu_2&=&\left(\begin{array}{ccc}
		- \cos(\pi x) \sin(\pi y)\\
		\sin(\pi x) \cos(\pi y)  
	\end{array}\right).\\
	p&=&\left \{\begin{array}{rcl}
		\cos(\pi x)\cos(\pi y) & \mbox{in} & \Omega_1 \\
		\cos(\pi x)\cos(\pi y)& \mbox{in} & \Omega_2
	\end{array}\right.,\quad
	A=\left \{\begin{array}{rcl}
		1 & \mbox{in} & \Omega_1 \\
		10& \mbox{in} & \Omega_2
	\end{array}\right..
\end{eqnarray*}
\end{example}

\begin{figure}[htbp]
	\centering
	\begin{minipage}[t]{1\textwidth}
		\centering
		\includegraphics[width=4cm]{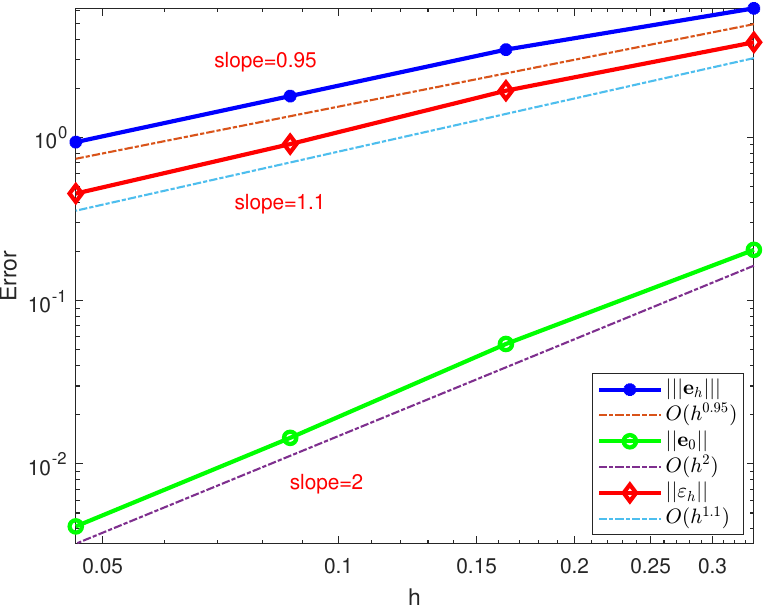}\hspace{0.5cm}
		\includegraphics[width=4cm]{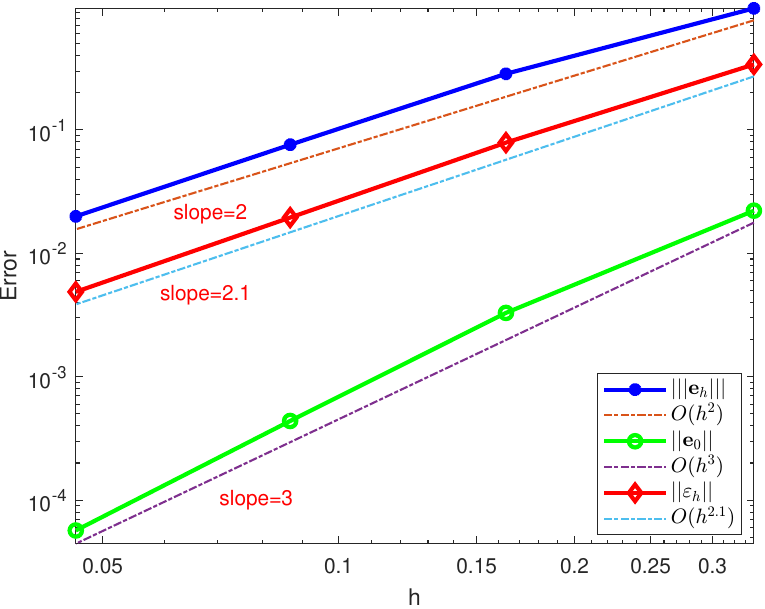}\hspace{0.5cm}
		\includegraphics[width=4cm]{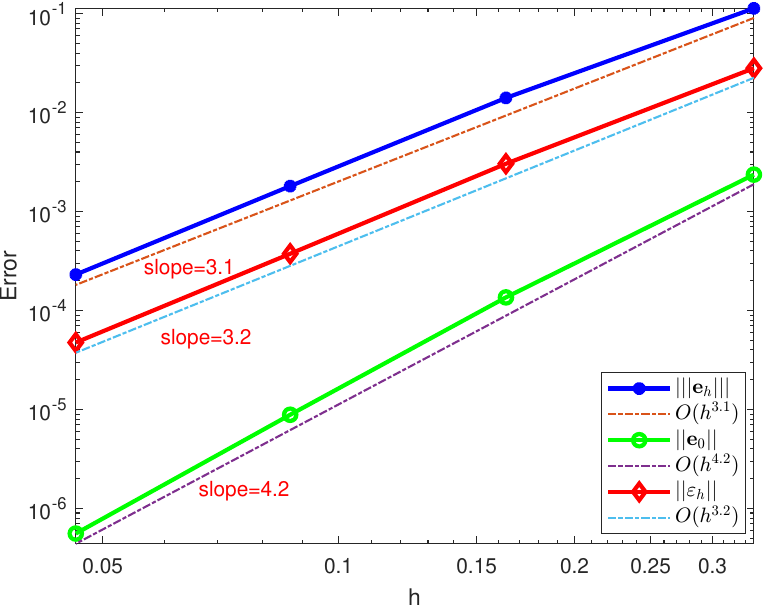}
		\caption{\small{ The numerical results for Example \ref{example2} on curved triangular meshes with $k=1$ (left), $2$ (middle), $3$ (right).}}\label{exam2-tri}
	\end{minipage}
\end{figure}

\begin{figure}[htbp]
	\centering
	\begin{minipage}[t]{1\textwidth}
		\centering
		\includegraphics[width=4cm]{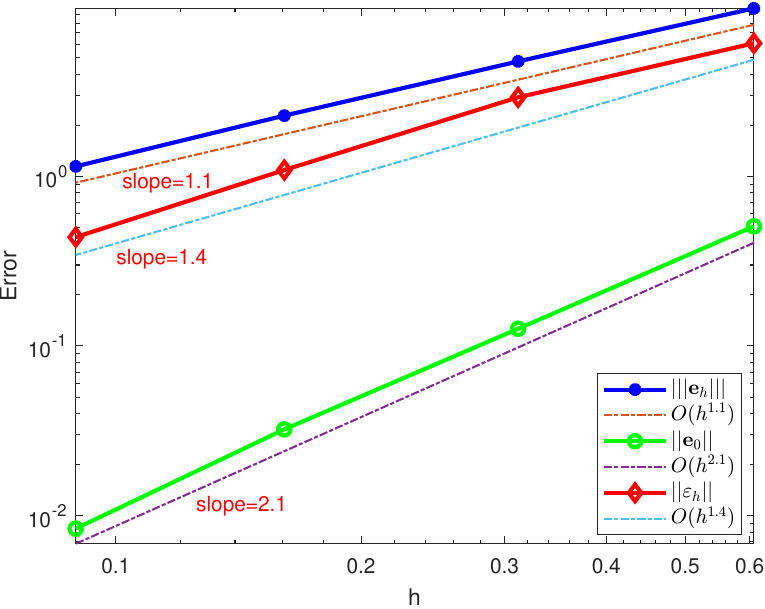}\hspace{0.5cm}
		\includegraphics[width=4cm]{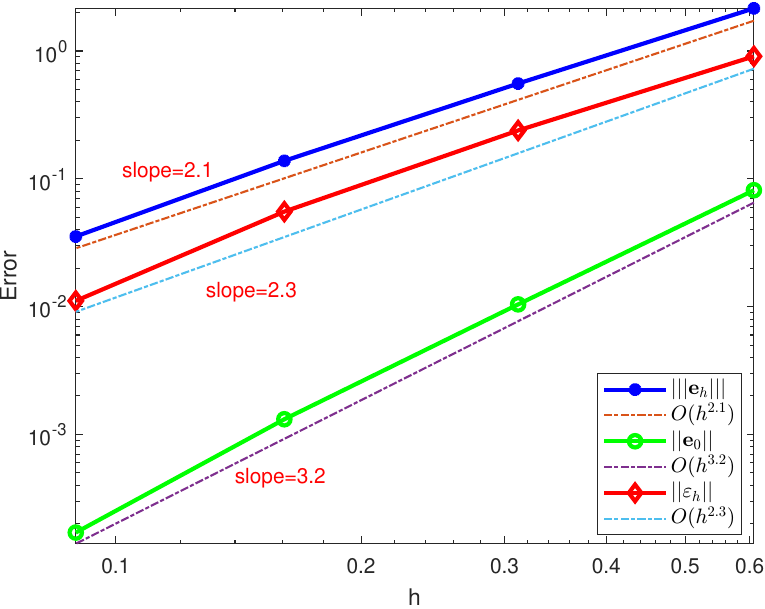}\hspace{0.5cm}
		\includegraphics[width=4cm]{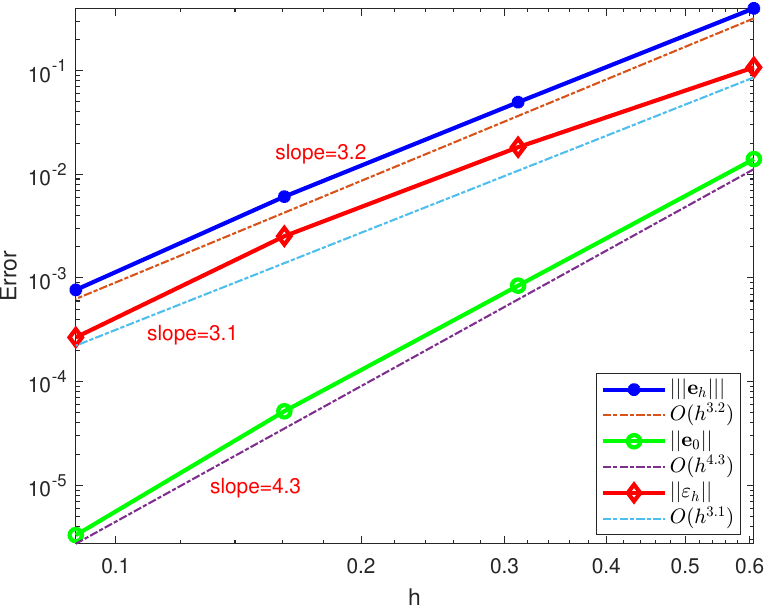}
		\caption{\small{ The numerical results for Example \ref{example2} on curved quadrilateral meshes with $k=1$ (left), $2$ (middle), $3$ (right).}}\label{exam2-rect}
	\end{minipage}
\end{figure}

\begin{figure*}[t!]
	\centering
	\begin{minipage}[t]{0.33\linewidth}
		\centering
		\includegraphics[width=1.1\linewidth]{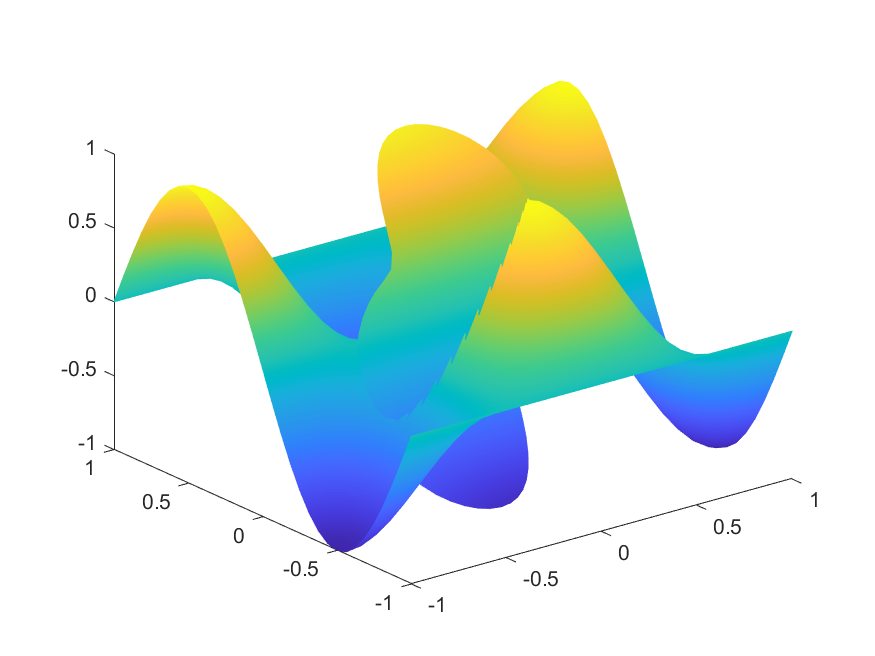}
	\end{minipage}%
	\begin{minipage}[t]{0.33\linewidth}
		\centering
		\includegraphics[width=1.1\linewidth]{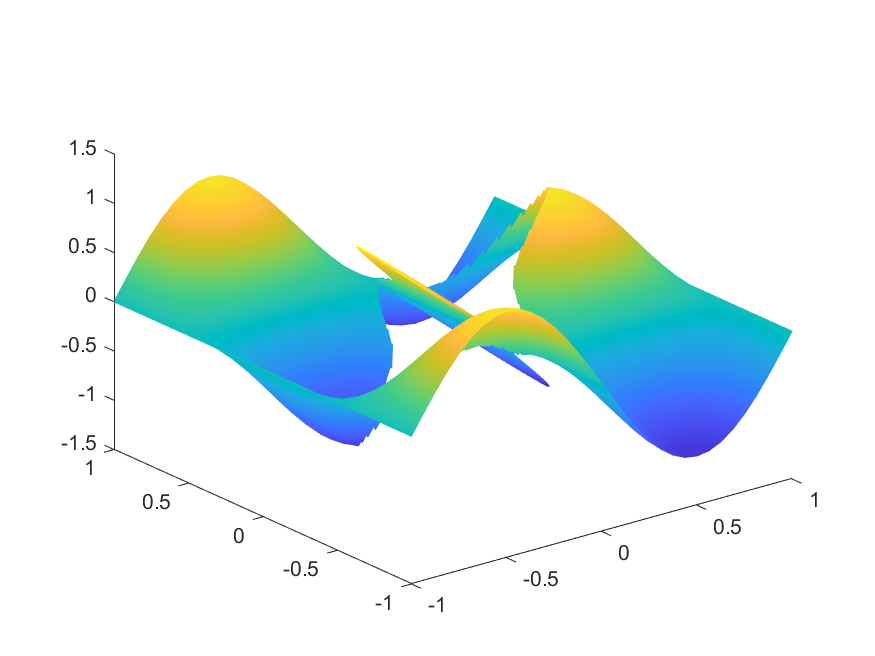}
	\end{minipage}
	\begin{minipage}[t]{0.33\linewidth}
		\centering
		\includegraphics[width=1.1\linewidth]{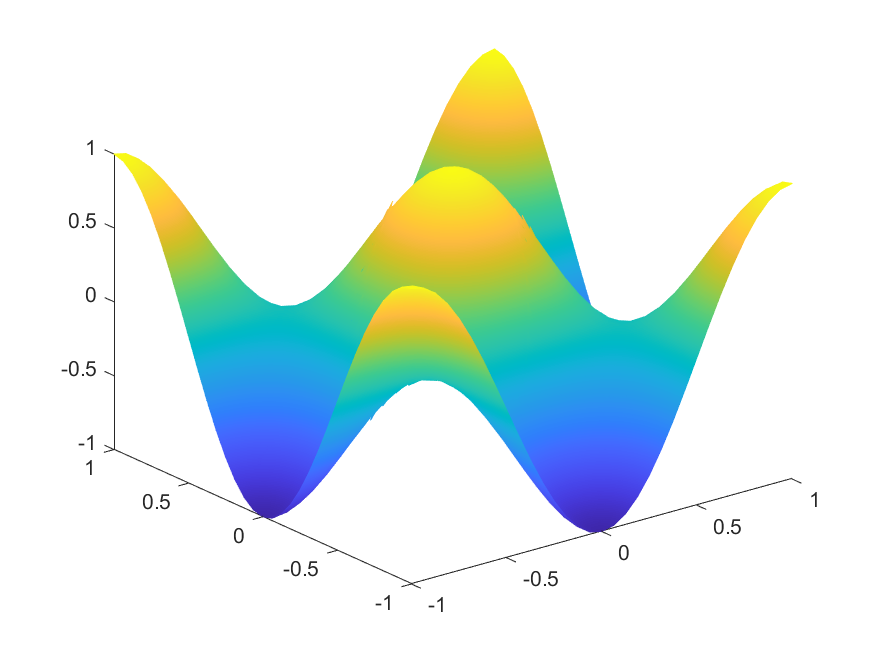}
	\end{minipage}
	\caption{The numerical solutions by $P_2$ WG element on triangular meshes level 3 in Example \ref{example2}. Left: The first component of $\bu_h$. Middle: The second component of $\bu_h$. Right: Pressure function $p_h$}
	\label{figure7.3}
\end{figure*}

In this example, we consider the interface problems with discontinuous velocity function $\bu$ and the viscosity coefficient $A$. The pressure function $p$ is continuous in the domain $\Omega$. The velocity function $\bu_h$ and pressure function $p_h$ are plotted in Figure \ref{figure7.3}, respectively. The numerical results are plotted in Figures \ref{exam2-tri} - \ref{exam2-rect}, by $P_1$ to $P_3$ WG elements on curved triangular meshes and curved quadrilateral meshes. The $P_k$ WG elements show the convergence orders $O(h^k)$ and $O(h^{k+1})$ for velocity functions in the energy norm and $L^2$ norm, respectively. For the pressure function, the $P_k$ WG elements achieve the convergence orders $O(h^k)$ in the $L^2$ norm. The orders of convergence are optimal in every case.


\section{Conclusion}In this paper, we use the weak Galerkin finite element method to deal with Stokes interface problems with curved interface. We present a weak Galerkin finite element numerical scheme with two values at the interface. Based on the WG scheme, we prove that numerical solutions converge to the exact solutions at the optimal rates. Additionally, the numerical results from our examples show the optimal convergence orders are obtained in both the energy norm and the $L^2$ norm on the triangular meshes and quadrilateral meshes. These results align with the theoretical analysis.

\section*{Declaration of competing interest}
The authors declare that they have no known competing financial interests or personal relationships that could have appeared to
influence the work reported in this paper.

\section*{Data availability}
The data that support the findings of this study are available from the corresponding author, upon reasonable request.

\section*{Acknowledgements}
This research was supported by the National Natural Science Foundation of China (grant No. 11901015, 12271208, 12001232, 12201246, 22341302), the National Key Research and Development Program of China (grant No. 2020YFA0713602, 2023YFA1008803), and the Key Laboratory of Symbolic Computation and Knowledge Engineering of Ministry of Education of China housed at Jilin University.


\begin{thebibliography}{10}
	
	\bibitem{FEMcurved1}
	{\sc R.~Aylwin and C.~Jerez-Hanckes}, {\em Finite-element domain approximation
		for {M}axwell variational problems on curved domains}, SIAM J. Numer. Anal.,
	61 (2023), pp.~1139--1171.
	
	\bibitem{pFEM}
	{\sc I.~Babuska, B.~A. Szabo, and I.~N. Katz}, {\em The p-version of the finite
		element method}, Siam. J. Numer. Anal., 18 (1981), pp.~515--545.
	
	\bibitem{VEMcurved1}
	{\sc L.~Beir\~{a}o~da Veiga, A.~Russo, and G.~Vacca}, {\em The virtual element
		method with curved edges}, ESAIM Math. Model. Numer. Anal., 53 (2019),
	pp.~375--404.
	
	\bibitem{DGcurved2}
	{\sc L.~Botti and D.~A. Di~Pietro}, {\em Assessment of hybrid high-order
		methods on curved meshes and comparison with discontinuous {G}alerkin
		methods}, J. Comput. Phys., 370 (2018), pp.~58--84.
	
	\bibitem{bookfiniteelementmethod1}
	{\sc S.~C. Brenner and L.~R. Scott}, {\em The mathematical theory of finite
		element methods}, Texts in Applied Mathematics, Springer-Verlag, New York,
	2002.
	
	\bibitem{bookfiniteelementmethod2}
	{\sc F.~Brezzi and M.~Fortin}, {\em Mixed and hybrid finite element methods},
	Springer Series in Computational Mathematics, Springer-Verlag, New York,
	1991.
	
	\bibitem{Wang_Stokes_Darcy}
	{\sc W.~Chen, F.~Wang, and Y.~Wang}, {\em Weak {G}alerkin method for the
		coupled {D}arcy-{S}tokes flow}, IMA J. Numer. Anal., 36 (2016), pp.~897--921.
	
	\bibitem{Stokesinterfacebackground4}
	{\sc E.~V. Chizhonkov}, {\em Numerical solution to a stokes interface problem},
	Comput. Math. and Math. Phys., 49 (2009), pp.~105--116.
	
	\bibitem{bookfiniteelementmethod3}
	{\sc M.~Crouzeix and P.-A. Raviart}, {\em Conforming and nonconforming finite
		element methods for solving the stationary {S}tokes equations. {I}}, Rev.
	Fran\c{c}aise Automat. Informat. Recherche Op\'{e}rationnelle S\'{e}r. Rouge,
	7 (1973), pp.~33--75.
	
	\bibitem{VEMcurved2}
	{\sc F.~Dassi, A.~Fumagalli, D.~Losapio, S.~Scial\`o, A.~Scotti, and G.~Vacca},
	{\em The mixed virtual element method on curved edges in two dimensions},
	Comput. Methods Appl. Mech. Engrg., 386 (2021), pp.~Paper No. 114098, 25.
	
	\bibitem{bookfiniteelementmethod5}
	{\sc V.~Giraud and P.~Raviart}, {\em Finite Element Methods for the
		Navier-Stokes Equations, Theory and Algorithms}, Springer Series in
	Computational Mathematics, Springer Berlin, 1986.
	
	\bibitem{Stokesinterfacebackground1}
	{\sc P.~P. Grinevich and M.~A. Olshanskii}, {\em An iterative method for the
		{S}tokes-type problem with variable viscosity}, SIAM J. Sci. Comput., 31
	(2009), pp.~3959--3978.
	
	\bibitem{guan2019weak}
	{\sc Q.~Guan}, {\em Weak galerkin finite element method for second order
		problems on curvilinear polytopal meshes with lipschitz continuous edges or
		faces}, 2019.
	
	\bibitem{bookfiniteelementmethod4}
	{\sc M.~D. Gunzburger}, {\em Finite element methods for viscous incompressible
		flows}, Computer Science and Scientific Computing, Academic Press, Inc.,
	Boston, MA, 1989.
	
	\bibitem{DGcurved3}
	{\sc H.~Huang, J.~Li, and J.~Yan}, {\em High order symmetric direct
		discontinuous {G}alerkin method for elliptic interface problems with fitted
		mesh}, J. Comput. Phys., 409 (2020), pp.~109301, 23.
	
	\bibitem{isogeometricmethods}
	{\sc P.~Kagan, A.~Fischer, and P.~Z. Bar-Yoseph}, {\em New b-spline finite
		element approach for geometrical design and mechanical analysis}, Int. J.
	Numer. Meth. Eng., 41 (1998), pp.~435--458.
	
	\bibitem{DGcurved1}
	{\sc E.~L. Kawecki}, {\em Finite element theory on curved domains with
		applications to discontinuous {G}alerkin finite element methods}, Numer.
	Methods Partial Differential Equations, 36 (2020), pp.~1492--1536.
	
	\bibitem{Wang_Curved_edges}
	{\sc Y.~Liu, W.~Chen, and Y.~Wang}, {\em A weak {G}alerkin mixed finite element
		method for second order elliptic equations on 2{D} curved domains}, Commun.
	Comput. Phys., 32 (2022), pp.~1094--1128.
	
	\bibitem{Stokesinterfacebackground6}
	{\sc L.~Mu}, {\em Numerical analysis for interface problems}, ProQuest LLC, Ann
	Arbor, MI, 2012.
	\newblock Thesis (Ph.D.)--University of Arkansas at Little Rock.
	
	\bibitem{mu_weak_2021}
	{\sc L.~Mu}, {\em Weak {G}alerkin
		finite element with curved edges}, J. Comput. Appl. Math., 381 (2021),
	p.~113038.
	
	\bibitem{WGBrinkman1}
	{\sc L.~Mu, J.~Wang, and X.~Ye}, {\em A stable numerical algorithm for the
		{B}rinkman equations by weak {G}alerkin finite element methods}, J. Comput.
	Phys., 273 (2014), pp.~327--342.
	
	\bibitem{ellipticinterface}
	{\sc L.~Mu, J.~Wang, X.~Ye, and S.~Zhao}, {\em A new weak {G}alerkin finite
		element method for elliptic interface problems}, J. Comput. Phys., 325
	(2016), pp.~157--173.
	
	\bibitem{Stokesinterfacebackground2}
	{\sc M.~A. Olshanskii and A.~Reusken}, {\em Analysis of a {S}tokes interface
		problem}, Numer. Math., 103 (2006), pp.~129--149.
	
	\bibitem{stokesDarcy}
	{\sc H.~Peng, Q.~Zhai, R.~Zhang, and S.~Zhang}, {\em A weak {G}alerkin-mixed
		finite element method for the {S}tokes-{D}arcy problem}, Sci. China Math., 64
	(2021), pp.~2357--2380.
	
	\bibitem{Stokesinterfacebackground5}
	{\sc D.~W. Schmid and Y.~Y. Podladchikov}, {\em Analytical solutions for
		deformable elliptical inclusions in general shear}, Geophys. J. Int., 155
	(2003), pp.~269--288.
	
	\bibitem{NEFEM}
	{\sc R.~Sevilla, S.~Fern\'{a}ndez-M\'{e}ndez, and A.~Huerta}, {\em
		N{URBS}-enhanced finite element method for {E}uler equations}, Internat. J.
	Numer. Methods Fluids, 57 (2008), pp.~1051--1069.
	
	\bibitem{reduceerror3}
	{\sc R.~Sevilla, S.~Fern\'{a}ndez-M\'{e}ndez, and A.~Huerta}, {\em
		N{URBS}-enhanced finite element method ({NEFEM})}, Internat. J. Numer.
	Methods Engrg., 76 (2008), pp.~56--83.
	
	\bibitem{DGcurved4}
	{\sc V.~Sobot\'{\i}kov\'{a}}, {\em Error analysis of a {DG} method employing
		ideal elements applied to a nonlinear convection-diffusion problem}, J.
	Numer. Math., 19 (2011), pp.~137--163.
	
	\bibitem{reduceerror1}
	{\sc B.~Szab\'{o} and I.~Babu\v{s}ka}, {\em Finite element analysis}, A
	Wiley-Interscience Publication, John Wiley \& Sons, Inc., New York, 1991.
	
	\bibitem{Stokesinterfacebackground3}
	{\sc V.~P. Trubitsyn, A.~A. Baranov, A.~Eyseev, and A.~Trubitsyn}, {\em Exact
		analytical solutions of the stokes equation for testing the equations of
		mantle convection with a variable viscosity}, IZV-PHYS SOLID EART+., 42
	(2006), pp.~537--545.
	
	\bibitem{elasticityequation}
	{\sc C.~Wang, J.~Wang, R.~Wang, and R.~Zhang}, {\em A locking-free weak
		{G}alerkin finite element method for elasticity problems in the primal
		formulation}, J. Comput. Appl. Math., 307 (2016), pp.~346--366.
	
	\bibitem{wang2013weak}
	{\sc J.~Wang and X.~Ye}, {\em A weak {G}alerkin finite element method for
		second-order elliptic problems}, J. Comput. Appl. Math., 241 (2013),
	pp.~103--115.
	
	\bibitem{elliptic_mix}
	{\sc J.~Wang and X.~Ye}, {\em A weak {G}alerkin
		mixed finite element method for second order elliptic problems}, Math. Comp.,
	83 (2014), pp.~2101--2126.
	
	\bibitem{WGStokes1}
	{\sc J.~Wang and X.~Ye}, {\em A weak {G}alerkin
		finite element method for the stokes equations}, Adv. Comput. Math., 42
	(2016), pp.~155--174.
	
	\bibitem{WGBrinkman2}
	{\sc X.~Wang, Q.~Zhai, and R.~Zhang}, {\em The weak {G}alerkin method for
		solving the incompressible {B}rinkman flow}, J. Comput. Appl. Math., 307
	(2016), pp.~13--24.
	
	\bibitem{reduceerror2}
	{\sc D.~Xue and L.~Demkowicz}, {\em Control of geometry induced error in {$hp$}
		finite element ({FE}) simulations. {I}. {E}valuation of {FE} error for
		curvilinear geometries}, Int. J. Numer. Anal. Model., 2 (2005), pp.~283--300.
	
	\bibitem{DGcurved5}
	{\sc F.~Yang and X.~Xie}, {\em An unfitted finite element method by direct
		extension for elliptic problems on domains with curved boundaries and
		interfaces}, J. Sci. Comput., 93 (2022), pp.~Paper No. 75, 26.
	
	\bibitem{WGStokes2}
	{\sc Q.~Zhai, R.~Zhang, and X.~Wang}, {\em A hybridized weak {G}alerkin finite
		element scheme for the {S}tokes equations}, Sci. China Math., 58 (2015),
	pp.~2455--2472.
	
	\bibitem{parabolic}
	{\sc H.~Zhang, Y.~Zou, Y.~Xu, Q.~Zhai, and H.~Yue}, {\em Weak {G}alerkin finite
		element method for second order parabolic equations}, Int. J. Numer. Anal.
	Model., 13 (2016), pp.~525--544.
	
	\bibitem{nonaffineisoparametric1}
	{\sc M.~Zl\'{a}mal}, {\em Curved elements in the finite element method. {I}},
	SIAM J. Numer. Anal., 10 (1973), pp.~229--240.
	
	\bibitem{nonaffineisoparametric2}
	{\sc M.~Zl\'{a}mal}, {\em Curved elements in
		the finite element method. {II}}, SIAM J. Numer. Anal., 11 (1974),
	pp.~347--362.
	
\end{thebibliography}

\newpage

\end{document}